\documentclass[12pt,twoside]{amsart}

\usepackage[english]{babel}
\usepackage[T1]{fontenc}
\usepackage{lmodern}
\usepackage{amsmath,amsthm,amssymb,amsfonts}
\usepackage{hyperref}
\usepackage[capitalise]{cleveref}
\usepackage[left=3cm,right=3cm,top=3cm,bottom=3cm,includeheadfoot]{geometry}
\usepackage{bbm}
\usepackage{xcolor}
\usepackage{enumerate,enumitem}
\usepackage{fancyhdr}
\usepackage{tikz-cd}
\usepackage[abbrev,nobysame]{amsrefs}

\usepackage[textsize=tiny]{todonotes}

\usepackage{xcolor}

\hypersetup{colorlinks=true,
	linkcolor=teal,
	citecolor=orange}

\newtheorem{satz}{Satz}
\numberwithin{satz}{section}
\newtheorem{cor}[satz]{Corollary}
\newtheorem{theorem}[satz]{Theorem}
\newtheorem{lemma}[satz]{Lemma} 

\newtheorem{conjecture}[satz]{Conjecture} 
\newtheorem{problem}[satz]{Problem}

\theoremstyle{definition}

\newtheorem{remark}[satz]{Remark}

\newtheorem{manualtheoreminner}{Theorem}
\newenvironment{manualtheorem}[1]{%
	\renewcommand\themanualtheoreminner{#1}%
	\manualtheoreminner
}{\endmanualtheoreminner}

\numberwithin{equation}{section}

\newcommand\R{\mathbb{R}}
\newcommand\C{\mathbb{C}}

\newcommand\N{\mathbb{N}}

\newcommand\vol{\operatorname{vol}}

\newcommand\id{\operatorname{id}}

\newcommand\spann{\operatorname{span}}

\newcommand\supp{\operatorname{supp}}
\newcommand\tr{\operatorname{tr}}

\newcommand\Rea{\operatorname{Re}}

\newcommand\Kc{\mathcal{K}}

\newcommand\SO{\operatorname{SO}}

\newcommand{\norm}[1]{\left\lVert#1\right\rVert}

\BibSpec{article}{%
	+{}{\PrintAuthors}  		{author}
	+{,}{ \textit}     		{title}
	+{,}{ }             		{journal}
	+{}{ \textbf}       		{volume}
	+{}{ \parenthesize} 		{date}
	+{}{, no. } 		{number}
	+{,}{ }      	      		{conference}
	+{,}{ }      	      		{book}
	+{,}{ }            		{pages}
	+{,}{ }            	 	{note}
	+{,}{ }            	 	{status}
	+{,}{  \texttt } {eprint}
	+{.}{}              {transition}
}

\BibSpec{book}{%
	+{}{\PrintAuthors}  {author}
	+{,}{ \textit}      {title}
	+{,}{ }             {publisher}
	+{,}{ }             {place}
	+{,}{ }             {date}
	+{.}{}              {transition}
}

\BibSpec{incollection}{%
	+{}  {\PrintAuthors}                {author}
	+{,} { \textit}                     {title}
	+{.} { }                            {part}
	+{:} { \textit}                     {subtitle}
	+{,} { \PrintContributions}         {contribution}
	+{,} { \PrintConference}            {conference}
	+{,} { }                            {booktitle}
	+{,} { pp.~}                        {pages}
	+{,}{ }       		{series}
	+{}{ \textbf}       		{volume}
	+{,} { }                            {publisher}
	+{,} { }                 {date}
	+{,} { }                            {status}
	+{,} { \PrintDOI}                   {doi}
	+{,} { available at \eprint}        {eprint}
	+{}  { \parenthesize}               {language}
	+{}  { \PrintTranslation}           {translation}
	+{;} { \PrintReprint}               {reprint}
	+{.} { }                            {note}
	+{.} {}                             {transition}
}

\usepackage{xcolor}
\usepackage[textsize=tiny]{todonotes}%colorinlistoftodos,prependcaption,textsize=tiny

\begin{document}
	\title[local log-Brunn-Minkowski inequality for bodies of revolution]{The local logarithmic Brunn-Minkowski inequality for bodies of revolution}
	\author{Luca Iffland}
	\address{Institut f\"ur Mathematik \newline%
		\indent Goethe-Universit\"at Frankfurt \newline%
		\indent Robert-Mayer-Str. 10, 60325 Frankfurt am Main, Germany}
	\email{iffland@math.uni-frankfurt.de}
	
	\begin{abstract}
		We prove the local logarithmic Brunn-Minkowski inequality for bodies of revolution. Furthermore, we give a generalization for one origin symmetric body of revolution and one body of revolution that does not need to be symmetric and restrict possible equality cases. The proof uses an operator theoretic approach together with the decomposition of spherical functions into isotypical components with respect to rotations around a fixed axis.
	\end{abstract}
	
	\maketitle
	
	\tableofcontents
	
	\thispagestyle{empty}
	
	\section{Introduction}
	
	Denote by $\Kc(\R^n)$ the set of compact convex sets in $\R^n$ which we refer to as convex bodies. The classical Brunn-Minkowski inequality states that for all $K,L\in\Kc(\R^n)$ and all $t\in[0,1]$
	\begin{align}\label{bm_ineq}
		\vol(tK+(1-t)L)^{\frac{1}{n}}\geq t\vol(K)^{\frac{1}{n}}+(1-t)\vol(L)^{\frac{1}{n}},
	\end{align}
	where $aK+bL=\{ax+by:x\in K,y\in L\}$ for $a,b \geq 0$ denotes the Minkowski addition of two sets in $\R^n$ and $\vol$ denotes the volume functional. This inequality, which goes back to the end of the 19th century, is one of the fundamentals of the Brunn-Minkowski theory and has an enormous amount of applications in convex geometry, stochastics and many other areas of mathematics. It is equivalent to the inequality
	\begin{align}\label{bm_ineq_weak}
		\vol(tK+(1-t)L)\geq \vol(K)^t\vol(L)^{1-t},
	\end{align}
	which might seem weaker than the above form as the geometric mean is always smaller than the arithmetic mean. However, the implication (\ref{bm_ineq_weak}) $\Rightarrow$ (\ref{bm_ineq}) can be seen by rescaling $K$ and $L$. We refer the reader to \cite{gardener} for more information about the Brunn-Minkowski inequality, its different forms and applications. In \cite{boeroeczkylutwagyangzhang}, Böröczky, Lutwak, Yang and Zhang conjectured a stronger version of the Brunn-Minkowski inequality, in which the left hand side in (\ref{bm_ineq_weak}) is replaced by the volume of some kind of geometric mean $K^tL^{1-t}$. For $K\in\Kc(\R^n)$ denote by
	\begin{align*}
		h_K:x\mapsto \max_{y\in K}\langle x,y\rangle
	\end{align*}
	its support function. The function $h_K$ can be viewed as either a function on the sphere or a 1-homogeneous function on $\R^n$. We will denote both the same and it will be always clear in the following, where the function is defined. It is well known that
	\begin{align*}
		h_{tK+(1-t)L}=th_K+(1-t)h_L,
	\end{align*}
	which suggests for the definition of $K^tL^{1-t}$ to replace the arithmetic mean of the two support functions of $K$ and $L$ by its geometric mean. However, $h_K^th_L^{1-t}$ is not, in general, the support function of a convex body. Instead, we define $K^tL^{1-t}$ as the Wulff shape of $h_K^th_L^{1-t}$, i.e.
	\begin{align*}
		K^tL^{1-t}=\bigcap_{x\in S^{n-1}}\left\{z\in\R^n:\langle x,z\rangle\leq h_K(x)^th_L(x)^{1-t} \right\}
	\end{align*}
	for convex bodies $K$ and $L$ that contain the origin in their interior. We say that $K\in\Kc(\R^n)$ is origin symmetric if $K=-K$. Denote by $\Kc_e(\R^n)$ the set of convex bodies, that are origin symmetric and contain the origin in their interior.
	
	The conjectured logarithmic Brunn-Minkowski inequality now states the following. 
	\begin{conjecture}\label{log_bm}
		For $K,L\in\Kc_e(\R^n)$ and all $t\in[0,1]$
		\begin{align}\label{log_bm_ineq}
			\vol(K^tL^{1-t})\geq \vol(K)^t\vol(L)^{1-t}.
		\end{align}
	\end{conjecture}
	For origin symmetric convex bodies, this is clearly stronger than the Brunn-Minkowski inequality as $K^tL^{1-t}\subseteq tK+(1-t)L$ by the inequality for the arithmetic and geometric mean. One can easily see that even in dimension $1$, the inequality (\ref{log_bm_ineq}) is not true for arbitrary convex bodies.
	
	To interpolate between the Brunn-Minkowski inequality and the logarithmic Brunn-Minkowski inequality, we can define the $L^p$-convex combination of $K,L\in\Kc_e(\R^n)$ for $p\neq0$ by
	\begin{align*}
		tK+_p(1-t)L = \bigcap_{x\in S^{n-1}}\left\{z\in\R^n:\langle x,z\rangle\leq \left(th_K(x)^p+(1-t)h_L(x)^p\right)^{\frac{1}{p}} \right\}
	\end{align*}
	and consider the following problem.
	\begin{problem}[$L^p$-Brunn-Minkowski problem]
		Under which conditions does
		\begin{align}\label{lp-brunn-minkowski-ineq}
			\vol(tK+_p(1-t)L)^{\frac{p}{n}}\geq t\vol(K)^{\frac{p}{n}}+(1-t)\vol(L)^{\frac{p}{n}}
		\end{align}
		hold?
	\end{problem}
	This question goes back to Firey \cite{firey} in 1962. Clearly, for $p=1$, the $L^p$-convex combination is just the classical Minkowski convex combination and hence (\ref{lp-brunn-minkowski-ineq}) is just the Brunn-Minkowski inequality. For $p\geq 1$ it turns out that the mapping $x\mapsto\left(th_K(x)^p+(1-t)h_L(x)^p\right)^{\frac{1}{p}}$ is indeed the support function of $tK+_p(1-t)L$. Firey could prove in \cite{firey} that the $L^p$-Brunn-Minkowski inequality (\ref{lp-brunn-minkowski-ineq}) for $p>1$ holds for all convex bodies $K$ and $L$.
	
	As $\left(th_K(x)^p+(1-t)h_L(x)^p\right)^{\frac{1}{p}}\to h_K(x)^th_L(x)^{1-t}$ for $p\to 0$, the geometric mean $K^tL^{1-t}$ is often also denoted as $tK+_0(1-t)L$. Therefore, Conjecture \ref{log_bm_ineq} is the natural extension of (\ref{lp-brunn-minkowski-ineq}) for $p=0$. For $0<p<1$ the $L^p$-Brunn-Minkowski inequality is known to not hold true for arbitrary convex bodies, which can be seen analogously to the case $p=0$. Böröczky, Lutwak, Yang and Zhang proved in \cite{boeroeczkylutwagyangzhang} when they conjectured the logarithmic Brunn-Minkowski inequality that it implies the $L^p$-Brunn-Minkowski inequality (\ref{lp-brunn-minkowski-ineq}) for all $p>0$ and all origin symmetric convex bodies $K$ and $L$. Furthermore they were able to show the logarithmic Brunn-Minkowski inequality in $\R^2$. For $p<0$, even under the assumption that $K$ and $L$ are origin symmetric, counterexamples to the $L^p$-Brunn-Minkowski inequality are known.
	
	Kolesnikov and Milman showed in \cite[Theorem 1.1]{kolesnikovmilman} that (\ref{lp-brunn-minkowski-ineq}) holds true for $p\in\left[1-cn^{-\frac{3}{2}},1\right]$ if the two bodies are origin symmetric and sufficiently close. They also proved that (\ref{log_bm_ineq}) holds true for origin symmetric bodies sufficiently close to the unit $\ell^q$-Ball for $q\geq 2$. Their results in \cite[Theorem 1.2, Theorem 1.3]{kolesnikovmilman} generalize the works \cite{colesantilivshyts, colesantilivshytsmarsiglietti} on the logarithmic Brunn-Minkowski inequality locally around the unit ball. In \cite{chenhuangliliu} the local result for $p\in\left[1-cn^{-\frac{3}{2}},1\right]$ was then extended to all origin symmetric convex bodies by Chen, Huang, Li and Liu. Böröczky and Kalantzopoulos \cite{boeroeczkykalantzopoulos} showed that (\ref{log_bm_ineq}) holds true if $K$ and $L$ are both symmetric under reflections along the same set of $n$ independent hyperplanes. The major ingredients in their proof are simplicial cones and the representations of Coxeter groups.
	
	Putterman \cite{putterman} showed that for any $p\in[0,1)$ the $L^p$-Brunn-Minkowski inequality is equivalent to the following local $L^p$-Brunn-Minkowski inequality confirming \cite[Conjecture 3.8]{kolesnikovmilman}. Denote by $V$ the mixed volume functional in $\R^n$ and by $S_K$ the surface area measure of the convex body $K$.
	
	\begin{conjecture}[local $L^p$-Brunn-Minkowski inequality]\label{loc_log_bm}
		For all $p\geq 0$ and $K,L\in\Kc_e(\R^n)$
		\begin{align}\label{loc_log_bm_ineq}
			\frac{V(L,K[n-1])^2}{\vol(K)}\geq \frac{n-1}{n-p}V(L,L,K[n-2])+\frac{1-p}{n(n-p)}\int_{S^{n-1}}\frac{h_L^2}{h_K}dS_K.
		\end{align}
	\end{conjecture}
	For $p=0$ the inequality (\ref{loc_log_bm_ineq}) is called the local logarithmic Brunn-Minkowski inequality. Van Handel \cite{vanhandel} showed that the local logarithmic Brunn-Minkowski inequality holds true if $K$ is a zonoid, i.e. the Hausdorff limit of finite Minkowski sums of line segments. His proof uses an observation of Kolesnikov and Milman \cite{kolesnikovmilman} and a variant of the Bochner method due to Shenfeld and himself \cite{shenfeldvanhandel} to transfer the statement to some inequality for an elliptic operator. The proof of the latter is then by induction on $n$. However, this is not known to imply the logarithmic Brunn-Minkowski inequality for zonoids. This is due to the fact that even if $K$ and $L$ are zonoids, the body $K^tL^{1-t}$, in general, is not a zonoid. In \cite{kolesnikovmilman}, Kolesnikov and Milman already used their elliptic operator to give a sufficient condition for the local logarithmic Brunn-Minkowski inequality to hold in the form of boundary Poincaré-type inequality. Furthermore, they confirm this sufficient condition in the case where $K$ is symmetric under reflections along $n$ independent hyperplanes. 
	
	In this work, we focus on Conjecture \ref{loc_log_bm} in the case where $K$ is a convex body of revolution, i.e. symmetric under rotations around a fixed axis. The treatment of classical problems in convex and integral geometry in the respective symmetry setting is an active area of research. In 1970, Firey first studied the Christoffel-Minkowski problem for convex bodies of revolution \cite{firey2} in the smooth setting. In the recent works \cite{braunerhofstaetterortegamoreno2, mussnigulivelli} the Christoffel-Minkowski problem for convex bodies of revolution is solved without smoothness. See also \cite{braunerhofstaetterortegamoreno, colesantimussnigludwig, hugmussnigulivelli, knoerr} for further work on integral geometric problems with symmetry under rotations around a fixed axis. In this article we give a detailed description of the elliptic operator appearing in \cite{kolesnikovmilman, vanhandel}, that is closely related to the local $L^p$-Brunn-Minkowski inequality, for the setting of convex bodies of revolution. We use this description to prove the local logarithmic Brunn-Minkowski inequality in this case. Though, this specific result is already contained in the result by Kolesnikov and Milman, our proof adds yet another perspective on the respective eigenvalue problem. In particular, we can prove the local logarithmic Brunn-Minkowski inequality for a body of revolution $K$ and a body $L$ that does not necessarily need to be symmetric around the origin. Furthermore we can restrict possible equality cases. 
	
	\begin{manualtheorem}{A}\label{local_log_bm_zonal}
		Let $K$ be an origin symmetric convex body of revolution with nonempty interior and let $L$ be an arbitrary convex body such that
		\begin{align}\label{center_of_mass_condition}
			\int_{S^{n-1}}\frac{xh_L(x)}{h_K(x)}dS_K(x)=0.
		\end{align}
		Then the pair $(K,L)$ satisfies (\ref{loc_log_bm_ineq}).
	\end{manualtheorem}
	
	The condition (\ref{center_of_mass_condition}) means geometrically that the center of mass of $L$ with respect to the positive measure $\frac{1}{h_K}dS_K$ lies in the origin. Therefore Theorem \ref{local_log_bm_zonal} implies that for any origin symmetric convex body of revolution $K$ and an arbitrary convex body $L$ there exists some translation $\tilde{L}=L+x$, where $x\in\R^n$ such that the pair $(K,\tilde{L})$ satisfies (\ref{loc_log_bm_ineq}). The same argument shows that for any convex body of revolution $K$ and any convex body $L$, there exists a translation of $\tilde{L}$ of $L$ such that the pair $(K,\tilde{L})$ does not satisfy (\ref{loc_log_bm_ineq}). 
	
	Furthermore we can give a restriction of the equality cases in Theorem \ref{local_log_bm_zonal} under some mild regularity assumptions on $K$.
	
	\begin{manualtheorem}{B}\label{equality_cases}
		Let $K$ and $L$ be as in Theorem \ref{local_log_bm_zonal} and assume that on $S^{n-1}$ the measure $\mathcal{H}^{n-1}$ is absolutely continuous with respect to $S_K$. Then if equality holds in (\ref{loc_log_bm_ineq}), $L$ must be a convex body of revolution.
	\end{manualtheorem}
	
	To prove Theorem \ref{local_log_bm_zonal} we will reformulate it into a spectral statement (Theorem \ref{spectral_gap_theorem}) for an appropriate elliptic operator $\mathcal{D}_K$ in the fashion of \cite{vanhandel}. This is a classical method, that has already been used in the early 20th century to prove the famous Alexandrov-Fenchel inequality, cf. \cite{aleksandrov1} and \cite{hilbert}, but also has many recent applications to geometric inequalities in integral geometry \cite{abardiawannerer, bernigkotrbatywannerer, kotrbatywannerer}. As we work in the symmetry setting, we will then use the general representation theory of the group $\SO(n-1)$ to decompose the domain of $\mathcal{D}_K$ into invariant subspaces $E_m$ for ${m\in\N_0}$ that are pairwise orthogonal. This decomposition, which is given explicitly in (\ref{direct_sum_L2}), allows us to transfer the spectral theorem to a countable set of eigenvalue problems for some Sturm-Liouville operators on $[-1,1]$ with a Jacobi-type weight. In Lemma \ref{lemma_explicit_op} these operators are described explicitly. 
	
	Such a decomposition has also been used in other areas of convex and integral geometry, cf. \cite{abardiawannerer} for some application towards Alexandrov-Fenchel-type inequalities for valuations that are invariant under the unitary group or \cite{kotrbatywannerer} for some Hodge-Riemann-type relations in an $\SO(n)$-invariant setting. However, we want to point out that in these works, the decomposition consists of finite dimensional vector spaces, whereas the spaces $E_m$ in our work are infinite dimensional. This is a reflection of the fact that in contrast to the groups treated in the mentioned works the group $\SO(n-1)$ does not act transitively on the unit sphere.
	
	The spectral problem on $E_0$ for even functions (Theorem \ref{theorem_eigenvalues_m=0}) is equivalent to Conjecture \ref{loc_log_bm} for two bodies of revolution that are symmetric under rotations around the same axis. Clearly two such bodies of revolution are automatically symmetric under reflections along the same set of $n$ independent hyperplanes and therefore the work by Böröczky and Kalantzopoulos \cite{boeroeczkykalantzopoulos} applies whenever $L$ is origin symmetric. To prove the stronger version without assuming origin symmetry (Theorem \ref{theorem_m=0}), we heavily rely on the explicit form of the operator restricted to $E_0$ in our setting and use a classical homotopy argument as in Alexandrov's second proof of the Alexandrov-Fenchel inequality \cite{aleksandrov1}. This will be the main content of Section \ref{section_m=0}.
	
	The spectral problem on the other $E_m$ will then follow from the result on $E_0$. This implication will be the content of Section \ref{section_m=1} and Section \ref{section_m>1}, where we treat the cases $m=1$ and $m\geq 2$ separately. This is essentially due to the fact that $E_1$ contains linear functions in contrast to the $E_m$ for $m\geq 2$. Theorem \ref{theorem_m=0} and Theorem \ref{theorem_m>0} together prove Theorem \ref{spectral_gap_theorem} and, therefore, Theorem \ref{local_log_bm_zonal}.
	
	In Section \ref{section_aux_results}, we prove some auxiliary results leading to the restriction of equality cases (Theorem \ref{equality_cases}) for sufficiently regular $K$. To do so, we use the fact that we can improve local logarithmic Brunn-Minkowski inequality proven in Theorem \ref{local_log_bm_zonal} by some stability term on the $E_m$ for $m\geq 1$ (Theorem \ref{stronger_loc_log_bm}). This will then imply that for a sufficiently regular body of revolution $K$ and any convex body $L$ such that equality holds in (\ref{loc_log_bm_ineq}), all projections of $h_L$ onto the $E_m$ must vanish for $m\geq 1$ and hence $K$ and $L$ must be bodies of revolution. 
	
	\section{Preliminaries}
	
	Fix an orthonormal basis $e_1,\dots,e_n$ of $\R^n$ and denote by $\langle\cdot,\cdot\rangle$ the standard euclidean inner product. In the whole paper we assume $n\geq 2$.
	
	Let $D$ be the standard flat connection on $\R^n$ and $\nabla$ the Levi-Civita connection of the standard metric on $S^{n-1}$. One can easily see, that for any $m$-homogeneous function $f\in C^2(\R^n)$ and $x\in S^{n-1}$ we have $Df(x)|_{T_xS^{n-1}}=\nabla f(x)$ as well as $D^2f(x)|_{T_xS^{n-1}}=\nabla^2f(x)+mf(x)\id$.
	
	In abuse of notation we also let $D:(\R^{(n-1)\times(n-1)})^{n-1}\to\R$ be the mixed discriminant, i.e. the symmetric multilinear extension of the determinant defined by
	\begin{align*}
		\det(\lambda_1A_1+\dots+\lambda_mA_m)=\sum_{i_1,\dots,i_{n-1}=1}^{m}D(A_{i_1},\dots,A_{i_{n-1}})\lambda_{i_1}\dots\lambda_{i_{n-1}}
	\end{align*}
	for all $\lambda_1,\dots,\lambda_m\in\R$ and $A_1,\dots,A_m\in\R^{(n-1)\times(n-1)}$. Furthermore Minkowski showed, that similarly the volume of a Minkowski linear combination of convex bodies is a homogeneous polynomial in the coefficients and therefore we can analogously define a symmetric Minkowski multilinear functional $V:\Kc(\R^n)^n\to \R$ by
	\begin{align*}
		\vol(\lambda_1K_1+\dots+\lambda_mK_m)=\sum_{i_1,\dots,i_n=1}^m V(K_{i_1},\dots,K_{i_n})\lambda_{i_1}\dots\lambda_{i_n}
	\end{align*}
	for all $\lambda_1,\dots,\lambda_m\geq 0$ and $K_1,\dots,K_m\in\Kc(\R^n)$. Clearly, $V(K,\dots,K)=\vol(K)$ by the homogeneity of the volume. The mixed volume can be extended to the subset of continuous functions on the sphere that are differences of support functions of convex bodies in the following way. Let $f=h_K-h_L$ for some $K,L\in\Kc(\R^n)$. Then set
	\begin{align*}
		V(f,C_1,\dots,C_{n-1})=V(K,C_1,\dots,C_{n-1})-V(L,C_1,\dots,C_{n-1})
	\end{align*}
	for any $C_1,\dots,C_{n-1}\in\Kc(\R^n)$. One easily checks that this is well defined. The extension of $V$ in all other entries is similar. We define the mixed area measure of the convex bodies $C_1,\dots,C_{n-1}$ to be the unique Borel measure $S_{C_1,\dots,C_{n-1}}$ on $S^{n-1}$ such that
	\begin{align*}
		V(f,C_1,\dots,C_{n-1})=\frac{1}{n}\int_{S^{n-1}} f dS_{C_1,\dots,C_{n-1}}
	\end{align*}
	for all differences of support functions $f$. The measure $S_K=S_{K,\dots,K}$ is called the surface area measure of $K$. Analogously to the mixed volume, the $C_i$ can be replaced by functions on the sphere by writing each function as a difference of support functions.
	
	We say $K\in \Kc(\R^n)$ is of the class $C^2_+$ if $h_K\in C^2(S^{n-1})$ and $D^2h_K>0$. For such a convex body $K$, the surface area measure $S_K$ is precisely the pullback of $\mathcal{H}^{n-1}$ on $\partial K$ under the inverse Gauss map. If $C_1,\dots,C_{n-1}\in\Kc(\R^n)$ are of class $C^2_+$ we have for the mixed area measure
	\begin{align*}
		dS_{C_1,\dots,C_{n-1}}=D(D^2h_{C_1},\dots,D^2h_{C_{n-1}})d\mathcal{H}^{n-1}.
	\end{align*}
	
	Now let $K\in\Kc(\R^n)$ be of the class $C^2_+$ with the origin in the interior. Then we define the measure $\mu_K$ on $S^{n-1}$ by
	\begin{align*}
		d\mu_K=\frac{1}{nh_K}dS_K=\frac{1}{nh_K}\det(D^2h_K)d\mathcal{H}^{n-1}.
	\end{align*}
	Furthermore for $f\in C^2(S^{n-1})$ define
	\begin{align*}
		\mathcal{D}_Kf=h_K\frac{D(D^2f,D^2h_K[n-2])}{\det(D^2h_K)}.
	\end{align*}
	By \cite{shenfeldvanhandel} the operator $\mathcal{D}_K$ extends to a second order uniformly elliptic symmetric differential operator on $L^2(S^{n-1},\mu_K)$. Its self-adjoint extension therefore has a discrete spectrum $\lambda_0>\lambda_1>\dots$ descending to $-\infty$ and the greatest eigenvalue $\lambda_0$ is simple and has a strictly positive eigenfunction. Obviously $h_K$ is an eigenfunction to the eigenvalue $1$ and hence $\lambda_0=1$. It is well known that the second eigenvalue is $\lambda_1=0$ with linear functions as eigenfunctions. This is equivalent to the Alexandrov-Fenchel inequality. The proof of this fact is known as Alexandrov's second proof of the Alexandrov-Fenchel inequality cf. \cite{aleksandrov1} and \cite{hilbert} for the classical proof or \cite{shenfeldvanhandel} for a recent simplified proof using the so called Bochner method. However, it will also follow from our work without using this particular result.
	
	Our aim is to prove the following spectral theorem.
	
	\begin{theorem}\label{spectral_gap_theorem}
		If $K$ is an origin symmetric body of revolution, then any eigenvalue $\lambda \notin  \{0,1\}$ of $\mathcal{D}_K$ satisfies 
		\begin{align*}
			\lambda \leq -\frac{1}{n-1}.
		\end{align*}
	\end{theorem}
	
	The proof of this theorem will be the main part of the present paper. Before proceeding, we want to use Theorem \ref{spectral_gap_theorem} to prove Theorem \ref{local_log_bm_zonal}. The argument closely follows the proof of Theorem 1.4 in \cite{vanhandel}.
	
	\begin{proof}[Proof of Theorem \ref{local_log_bm_zonal}]
		
		Assume $K$ is of the class $C_2^+$ and denote by $\langle\cdot,\cdot\rangle$ the inner product of $L^2(S^{n-1},\mu_K)$. By Theorem \ref{spectral_gap_theorem} for every $f\in C^2(S^{n-1})$ that is orthogonal to $h_K$ and linear functions we have
		\begin{align}\label{operator_ineq}
			\langle \mathcal{D}_Kf,f\rangle+\frac{1}{n-1}\langle f,f\rangle \leq 0.
		\end{align}
		For any origin symmetric convex body $L$ of class $C_+^2$ we set $f=h_L-\frac{\langle h_L,h_K\rangle}{\langle h_K,h_K\rangle}h_K$. Using the above inequality together with the identities
		\begin{align*}
			\langle \mathcal{D}_Kh_L,h_L\rangle =& \ V(L,L,K[n-2]), \\
			\langle h_L,h_L\rangle =& \ \frac{1}{n}\int_{S^{n-1}}\frac{h_L^2}{h_K}dS_K, \\
			\langle h_K,h_K\rangle=& \ \vol(K), \\
			\langle \mathcal{D}_Kh_L,h_K\rangle =& \ \langle h_L,h_K\rangle = V(L,K[n-1])
		\end{align*}
		yields (\ref{loc_log_bm_ineq}). Finally for $K$ and $L$ not necessarily of class $C_+^2$, approximating $K$ by bodies of revolution of the class $C_+^2$, $L$ by convex bodies of the class $C_+^2$ satisfying (\ref{center_of_mass_condition}) and using that (\ref{loc_log_bm_ineq}) is continuous with respect to the Hausdorff topology proves the theorem in full generality.		
	\end{proof} 
	
	\section{The elliptic operator $\mathcal{D}_K$}
	
	We define the following two differential operators that will describe the eigenvalues of the hessian of a zonal function, i.e. a function that is invariant under rotations around a fixed axis. For $\eta\in C^2[-1,1]$ set
	\begin{align*}
		\mathcal{A}_1\eta(t) & =-t\eta'(t)+\eta(t),\\
		\mathcal{A}_2\eta(t) &=(1-t^2)\eta''(t)+\mathcal{A}_1\eta(t).
	\end{align*}
	
	We will often use the relation
	\begin{align*}
		\frac{d}{dt} \mathcal A_1\eta(t)=-t\eta''(t)=-\frac{t}{1-t^2} (\mathcal A_2 \eta(t)-\mathcal A_1 \eta(t))
	\end{align*}
	and hence
	\begin{align*}
		\frac{d}{dt}\sqrt{1-t^2}\mathcal{A}_1\eta(t)=-\frac{t\mathcal{A}_2\eta(t)}{\sqrt{1-t^2}}.
	\end{align*}
	
	\begin{lemma}\label{zonal_hesssian}
		For $\eta\in C^2[-1,1]$ and $h\in C^2(S^{n-1})$ given by $h(x)=\eta(x_n)$ we have
		\begin{align*}
			D^2h(x)\big|_{T_xS^{n-1}} = \eta''(x_n)v(x) v(x)^T + \mathcal{A}_1\eta(x_n)\id
		\end{align*}
		for all $x\in S^{n-1}$, where $v(x)=P_{x^\perp}e_n$ is the orthogonal projection of $e_n$ to the tangent space of $S^{n-1}$ at $x$.
	\end{lemma}
	
	\begin{proof}
		Taking the second derivative of the 1-homogeneous extension of $h$, given by $x\mapsto |x|\eta\left(\frac{x_n}{|x|}\right)$, and restricting it to the tangent space of $S^{n-1}$ this is a direct computation, c.f. \cite[(5.3)]{ortegamorenoschuster}.
	\end{proof}
	
	For the rest of the paper if not stated otherwise we assume $K$ to be a body of revolution of the class $C^2_+$ with the $x_n$-axis as axis of symmetry and hence $h_K(x)=\eta(x_n)$ for some $\eta\in C^2[-1,1]$ by \cite[Theorem 2.17]{braunerhofstaetterortegamoreno}. Now observe that $|v(x)|=\sqrt{1-x_n^2}$ for $x\in S^{n-1}$. By Lemma \ref{zonal_hesssian} for $x\neq \pm e_n$ the map $D^2h_K(x)$ has the eigenvalue $\mathcal{A}_2\eta(x_n)$ with eigenspace the line spanned by $v(x)$ and the eigenvalue $\mathcal{A}_1\eta(x_n)$ with eigenspace $v(x)^\perp$, hence of multiplicity $(n-2)$. As $v(\pm e_n)=0$ we have $D^2h(\pm e_n)=\mathcal{A}_1\eta(\pm 1)\id=\mathcal{A}_2\eta(\pm 1)\id$. In particular the assumption that $K$ is of the class $C^2_+$ implies that $\mathcal{A}_2\eta(t)>0$ and $\mathcal{A}_1\eta(t)>0$ for all $t\in [-1,1]$. As, by assumption, the origin lies in the interior of $K$, we also have $\eta(t)>0$ for all $t\in[-1,1]$. 
	
	\begin{lemma}\label{lemma_explicit_mixed_discriminant}
		We have
		\begin{align*}
			d\mu_K(x)=\frac{\mathcal{A}_2\eta(x_n)\mathcal{A}_1\eta(x_n)^{n-2}}{n\eta(x_n)}d\mathcal{H}^{n-1}(x)
		\end{align*}
		as well as
		\begin{align*}
			\mathcal{D}_K f(x)=&\frac{\eta(x_n)}{n-1}\frac{\mathcal{A}_2\eta(x_n)\tr(D^2f(x))-\eta''(x_n)\langle v(x),D^2f(x)v(x)\rangle}{\mathcal{A}_2\eta(x_n)\mathcal{A}_1\eta(x_n)}
		\end{align*}
		for all $f\in C^2(S^{n-1})$.
	\end{lemma}
	
	\begin{proof}
		By our knowledge of the eigenvalues of $D^2h_K(x)$ we have 
		\begin{align*}
			\det(D^2h_K(x))=\mathcal{A}_2\eta(x_n)\mathcal{A}_1\eta(x_n)^{n-2},
		\end{align*}
		which proves the first identity.
		
		For $n=2$ we directly see that $\mathcal{D}_Kf(x)=\eta(x_n)\frac{D^2f(x)}{\mathcal{A}_2\eta(x_n)}$ which coincides with the stated form. Now assume $n\geq 3$. As for any $A\in\R^{(n-1)\times(n-1)}$ and $v\in\R^{n-1}$ we have $D(A,\id[n-2])=\frac{1}{n-1}\tr(A)$ as well as $D(A,vv^T,\id[n-3])=\frac{1}{(n-1)(n-2)}(\tr(A)|v|^2-\langle v,Av\rangle)$, we get
		\begin{align*}
			&D(D^2f(x),D^2h_K(x)[n-2])\\
			&\quad=\mathcal{A}_1\eta(x_n)^{n-2}D(D^2f(x),\id[n-2])\\
			& \quad \quad + (n-2)\eta''(x_n)\mathcal{A}_1\eta(x_n)^{n-3}D(D^2f(x),v(x) v(x)^T,\id[n-3]) \\
			& \quad = \frac{1}{n-1}\mathcal{A}_2\eta(x_n)\mathcal{A}_1\eta(x_n)^{n-3}\tr(D^2f(x)) \\
			& \quad \quad -\frac{1}{n-1}\eta''(x_n)\mathcal{A}_1\eta(x_n)^{n-3}\langle v(x),D^2f(x)v(x)\rangle.
		\end{align*}
	\end{proof}
	
	For our work we need some results about products of zonal functions with spherical harmonics, which can be shown by calculation in cylindrical coordinates. Denote by $\mathcal{H}_m^n$ the space of spherical harmonics of degree $m$, that is, the space of restrictions of $m$-homogeneous harmonic polynomials to $S^{n-1}$. Similarly as for the support functions of convex bodies, we will not distinguish between a spherical harmonic and its $m$-homogeneous extension to $\R^n$.
	
	In the following we denote 
	\begin{align*}
		w_{m,n}(t)=\frac{\mathcal{A}_2\eta(t)\mathcal{A}_1\eta(t)^{n-2}}{n\eta(t)}(1-t^2)^{m+\frac{n-3}{2}}.
	\end{align*}
	
	\begin{lemma}\label{lemma_Lp}
		Let $f_1,f_2:S^{n-1} \to\R$ be given by $$f_j(x)=\sum_{i=1}^Nh_{i,j}(x_1,\dots,x_{n-1})\overline{f}_{i,j}(x_n)$$ for $j=1,2$, where for $i=1,\dots,N$ the functions $\overline{f}_{i,j}:[-1,1]\to\R$ are measurable and the functions $h_{i,j}: \R^{n-1} \to \R$ form an orthonormal basis of $\mathcal{H}^{n-1}_m$ with respect to the $L^2$-inner product on $S^{n-2}$. Then 
		\begin{align*}
			\int_{S^{n-1}} f_1f_2d\mu_K = \sum_{i=1}^N \int_{-1}^1 w_{m,n}(t)\overline{f}_{i,1}(t)\overline{f}_{i,2}(t)dt.
		\end{align*}
		In particular, $f_j\in L^2(S^{n-1})$ for $j=1,2$ if and only if
		\begin{align*}
			\overline{f}_{i,j}(t)\in L^2\left([-1,1],w_{m,n}(t)dt\right)
		\end{align*}
		for all $i=1,\dots,N$.
	\end{lemma}
	
	\begin{proof}
		We use Lemma \ref{lemma_explicit_mixed_discriminant} as well as the parametrization $x=(\sqrt{1-t^2}y,t)$ for $t\in[-1,1]$ and $y\in S^{n-2}$ to get
		\begin{align*}
			&\int_{S^{n-1}} f_1f_2 d\mu_K \\
			& \quad =\frac{1}{n}\int_{S^{n-1}} \frac{\mathcal{A}_2\eta(x_n)\mathcal{A}_1\eta(x_n)^{n-2}}{\eta(x_n)}f_1(x)f_2(x)d\mathcal{H}^{n-1}(x) \\
			& \quad = \frac{1}{n}\sum_{i,i'=1}^N\int_{-1}^1 \frac{\mathcal{A}_2\eta(t)\mathcal{A}_1\eta(t)^{n-2}}{\eta(t)}(1-t^2)^{\frac{n-3}{2}}\overline{f}_{i,1}(t)\overline{f}_{i',2}(t) \\
			& \hspace{5cm} \int_{S^{n-2}} h_{i,1}(\sqrt{1-t^2}y)h_{i',2}(\sqrt{1-t^2}y) d\mathcal{H}^{n-2}(y) dt \\
			& \quad  =\frac{1}{n} \ \sum_{i=1}^N\int_{-1}^1 \frac{\mathcal{A}_2\eta(t)\mathcal{A}_1\eta(t)^{n-2}}{\eta(t)}(1-t^2)^{m+\frac{n-3}{2}} \overline{f}_{i,1}(t)\overline{f}_{i,2}(t)dt
		\end{align*}
		by Fubini's theorem and the orthonormality of the $h_i$. This proves the first claim. Now by definition $f_j\in L^2 (S^{n-1})$ if and only if $\int_{S^{n-1}}f_j^2d\mu_K<\infty$ and hence the second claim follows.
	\end{proof}
	
	\begin{lemma}\label{lemma_Cinfty}
		Let $$f=\sum_{i=1}^N h_{i}(x_1,\dots,x_{n-1})\overline{f}_{i}(x_n)$$ be as in the previous lemma, $m>0$ and assume $f\in C^2(S^{n-1})$. Then for all $i=1,\dots,N$ we have $\overline{f}_i\in C^2(-1,1)$ and the boundary conditions 
		\begin{align}\label{dirichlet_boundary_conditions}
			\limsup\limits_{t\to\pm 1} (1-t^2)^{\frac{m-1}{2}}\left|\overline{f}_i(t)\right|<\infty
		\end{align}
		as well as
		\begin{align}\label{neumann_boundary_conditions}
			\limsup\limits_{t\to\pm 1} (1-t^2)^{\frac{m+1}{2}}\left|\overline{f}_i'(t)\right|<\infty.
		\end{align}
	\end{lemma}
	
	\begin{remark}\label{remark_Cinfty}
		For $m=0$, any homogeneous polynomial of degree $m$ is constant. In this case, a function $f:S^{n-1}\to \R$ given by $f(x)=\overline{f}(x_n)$ is $C^k$ for some $k\in\N_0\cup\{\infty\}$ if and only if $\overline{f}\in C^k[-1,1]$. A proof of this fact for $k=\infty$ can be found in \cite[Theorem 2.17]{braunerhofstaetterortegamoreno}, for $k<\infty$ it is analogous.
	\end{remark}
	
	\begin{proof}[Proof of Lemma \ref{lemma_Cinfty}]
		As the $h_i$ are linearly independent, there exist $y_1,\dots,y_N\in S^{n-2}$ such that the matrix $(h_i(y_j))_{i,j=1}^N$ is invertible. Now by definition we have
		\begin{align}\label{linear_system_dirichlet_bc}
			f(\sqrt{1-t^2}y_j,t)=\sum_{i=1}^N h_i(\sqrt{1-t^2}y_j)\overline{f}_i(t)=(1-t^2)^{\frac{m}{2}}\sum_{i=1}^N h_i(y_j)\overline{f}_i(t).
		\end{align}
		We multiply with the inverse $((h_i(y_j))_{i,j=1}^N)^{-1}=(\tilde{h}_{ji})_{i,j=1}^N$ to get 
		\begin{align*}
			\sum_{j=1}^N f(\sqrt{1-t^2}y_j,t) \tilde{h}_{ji}=(1-t^2)^{\frac{m}{2}} \overline{f}_i(t),
		\end{align*}
		which shows that $\overline{f}_i \in C^2(-1,1)$. Taking the limit $t\to\pm 1$ we get
		\begin{align*}
			f(\pm e_n)\sum_{j=1}^N \tilde{h}_{ji}=\lim\limits_{t\to\pm 1} \sum_{j=1}^N \tilde{h}_{ji}f(\sqrt{1-t^2}y_j,t)=\lim\limits_{t\to\pm 1}(1-t^2)^{\frac{m}{2}}\overline{f}_i(t).
		\end{align*}
		We now show that $f(\pm e_n)=0$. To this end denote $a_i^\pm=\lim\limits_{t\to\pm 1}(1-t^2)^{\frac{m}{2}}\overline{f}_i(t)$ and observe that $h=\sum_{i=1}^Na_i^\pm h_i$ is again a spherical harmonic of degree $m$. By the assumption $m>0$ we have that $h$ is orthogonal to constants and hence there exists $y\in S^{n-2}$ such that $h(y)=0$. Then with (\ref{linear_system_dirichlet_bc}) we get
		\begin{align*}
			f(\pm e_n)=\lim\limits_{t\to\pm 1} f(\sqrt{1-t^2}y,t)=\sum_{i=1}^Na_i^\pm h_i(y)=h(y)=0.
		\end{align*}
		Finally by the assumption, $f$ is Lipschitz continuous. Now assume $t\geq 0$. For any $y\in S^{n-2}$ we have
		\begin{align*}
			|(\sqrt{1-t^2}y,t)- e_n|^2=2(1-t)\leq 2(1-t^2).
		\end{align*}
		With that we get
		\begin{align*}
			(1-t^2)^{\frac{m-1}{2}}|\overline{f}_i(t)|=& \ \frac{\sum_{j=1}^N \left(f(\sqrt{1-t^2}y_j,t)-f(e_n)\right) \tilde{h}_{ji}}{\sqrt{1-t^2}} \\
			\leq & \  2\sum_{j=1}^N \frac{\left|f(\sqrt{1-t^2}y_j,t)-f(e_n)\right|}{|(\sqrt{1-t^2}y_j,t)- e_n|}\tilde{h}_{ji}.
		\end{align*}
		By the Lipschitz continuity the right hand side is bounded and hence $t\to 1$ shows (\ref{dirichlet_boundary_conditions}) at $1$. At $-1$ the argument is the same.
		
		For the derivative we consider the $0$-homogeneous extension of $f$ given by
		\begin{align*}
			f(x)=\sum_{i=1}^N h_i\left(\frac{(x_1,\dots,x_{n-1})}{|x|}\right)\overline{f}_i\left(\frac{x_n}{|x|}\right)
		\end{align*}
		and calculate that for $x\in S^{n-1}$
		\begin{align*}
			Df(x)=&\sum_{i=1}^N h_i(x_1,\dots,x_{n-1})\overline{f}'_i(x_n)(e_n-x_nx)+\overline{f}_i(x_n)(\id-xx^T)Dh_i(x_1,\dots,x_{n-1}) \\
			=& \sum_{i=1}^N h_i(x_1,\dots,x_{n-1})\left(\overline{f}'_i(x_n)(e_n-x_nx)-m\overline{f}_i(x_n)x\right) + \overline{f}_i(x_n)Dh_i(x_1,\dots,x_{n-1}).
		\end{align*}
		Now consider the continuous vector field $X$ on $S^{n-1}\setminus\{\pm e_n\}$ given by $X_x=(1-x_n^2)^{-\frac{1}{2}}P_{x^\perp}e_n$. As $|P_{x^\perp}e_n|=\sqrt{1-x_n^2}$, the vector field $X$ is bounded around $\pm e_n$. Now
		\begin{align*}
			Df(\sqrt{1-t^2}y,t)X_{(\sqrt{1-t^2}y,t)}=&(1-t^2)^{\frac{m-1}{2}}\sum_{i=1}^N h_i(y)\left((1-t^2) \overline{f}'_i(t)-mt\overline{f}_i(t)\right).
		\end{align*}
		As $\limsup\limits_{x\to \pm e_n}|Df(x)X_x|<\infty$, applying the same method as above and using (\ref{dirichlet_boundary_conditions}) gives (\ref{neumann_boundary_conditions}).
	\end{proof}
	
	We now briefly recall the decomposition of spherical functions into irreducible $\SO(n)$ representations. This decomposition was widely used in convex geometry, see \cite[Appendix]{schneider} for an introduction as well as an overview of applications in convex geometry and further literature. For $f:S^{n-1}\to \R$ the action of $\SO(n)$ on $f$ is given by $(gf)(x)=f(g^{-1}x)$. It is well known that
	\begin{align*}
		L^2(S^{n-1})=\bigoplus_{m=0}^\infty \mathcal{H}_m^n,
	\end{align*}
	where the sum is orthogonal. For $n\geq3$, each $\mathcal{H}_m^n$ is an irreducible $\SO(n)$ representation with highest weight $(m,0,\dots,0)$, for $n=2$, the space $\mathcal{H}^n_m$ splits into two one dimensional $\SO(n)$ representations with weights $(m,0,\dots,0)$ and $(0,\dots,0,-m)$. For $n=1$, the only homogeneous harmonic polynomials on $\R$ are constants and linear functions. Hence $\mathcal{H}_m^1$ is only nontrivial for $m=0,1$ and one dimensional in these cases. Now define
	\begin{align*}
		E_m=\mathcal{H}_m^{n-1}\otimes L^2\left([-1,1],w_{m,n}(t)dt\right).
	\end{align*}
	We equip $E_m$ with the inner product given by
	\begin{align*}
		\langle h_1\otimes \overline{f}_1,h_2\otimes \overline{f}_2\rangle = \langle h_1,h_2\rangle_{L^2(S^{n-2})} \langle \overline{f}_1,\overline{f}_2\rangle_{m,n},
	\end{align*}
	where $\langle\cdot,\cdot\rangle_{m,n}$ denotes the $L^2$-inner product on $[-1,1]$ with respect to the measure $w_{m,n}(t)dt$. For $h\in \mathcal{H}_m^{n-1}$ and $\overline{f}\in L^2\left([-1,1],w_{m,n}(t)dt\right)$ we denote 
	\begin{align*}
		(h\otimes \overline{f})(x)=h(x_1,\dots,x_{n-1})\overline{f}(x_n)
	\end{align*}
	for $x\in S^{n-1}$. Extending this by linearity together with Lemma \ref{lemma_Lp} yields an isometric embedding $E_m\hookrightarrow L^2(S^{n-1},\mu_K)$. By the orthogonality of the $\mathcal{H}_m^{n-1}$, the images of the $E_m$ are pairwise orthogonal for $m\in \N_0$. We claim that
	\begin{align}\label{direct_sum_L2}
		L^2(S^{n-1},\mu_K)=\bigoplus_{m=0}^\infty E_m,
	\end{align}
	where the sum denotes the orthogonal sum of Hilbert spaces. As the right hand side is clearly closed and contained in the left hand side, it is sufficient to show that it is dense. To show this consider the subset of $C(S^{n-1})$ given by finite linear combinations of functions of the form $x\mapsto h(x_1,\dots,x_{n-1})\overline{f}(x_n)$, where $h\in \mathcal{H}^{n-1}_m$ for some $m\in\N_0$ and $\overline{f}\in C[-1,1]$. By the Stone-Weierstraß theorem this subset is dense in $C(S^{n-1})$. As $C(S^{n-1})$ is dense in $L^2(S^{n-1},\mu_K)$, the claim is proven.
	
	For $n>3$, each $E_m$ is $\SO(n-1)$ isotypical with highest weight $(m,0,\dots,0)$. For a suitable ordering of the roots, the highest weight vectors are precisely the functions of the form $x\mapsto z_{\overline{1}}^m\overline{f}(x_n)$, where $z_{\overline{1}}=x_1-\sqrt{-1}x_2$. For $n=3$, $E_m$ is a sum of isotypical components with weights $(m,0)$ and $(0,-m)$ if $m>0$. We refer the reader to the books by Bröcker and Tom Dieck \cite{broeckertomdieck} and Fulton and Harris \cite{fultonharris} for further material on this topic. To treat the action of $\mathcal{D}_K$ on the decomposition (\ref{direct_sum_L2}), we consider the orthogonal projections $\pi_m$ onto $E_m$. It is well known that if $n\neq 3$ and $f\in L^2(S^{n-1},\mu_K)$ then
	\begin{align*}
		\pi_mf=d_m\int_{\SO(n-1)}\overline{\chi_m(g)}gf dg,
	\end{align*}
	where $d_m$ is the dimension of the irreducible $\SO(n-1)$ representation with weight $(m,0,\dots,0)$, $\chi_m$ is the character of the respective irreducible representation and $dg$ denotes the Haar measure on $\SO(n-1)$. For $n=3$, the projection $\pi_m$ is the sum of two projections of the above form. Clearly this implies that $\pi_m$ is continuous. Furthermore the restriction $\pi_m:C^k(S^{n-1})\to E_m\cap C^k(S^{n-1})$ is well defined and continuous in the $C^k$ topology.
	
	As $\mathcal{D}_K$ is uniformly elliptic and symmetric, its self-adjoint extension is defined on the Sobolev space $H^2(S^{n-1},\mu_K)$ i.e. the space of functions in $L^2(S^{n-1},\mu_K)$ that have weak derivatives up to second order in $L^2(S^{n-1},\mu_K)$, which contains $C^2$ functions as a dense subset. The $\SO(n-1)$-invariance of $h_K$ implies that $\mathcal{D}_K$ is $\SO(n-1)$-equivariant and hence it restricts to a sequence of self adjoint operators on $\mathcal{D}_m=\mathcal{D}_K|_{E_m}$. For any $f\in H^2(S^{n-1},\mu_K)$ we have 
	\begin{align*}
		\pi_m \mathcal{D}_Kf= \mathcal{D}_m(\pi_mf),
	\end{align*}
	as the above description of $\pi_m$ implies that it is continuous in the $H^2$ topology. This implies that any $\lambda\in\R$ is an eigenvalue of $\mathcal{D}_K$ if and only if it is an eigenvalue of at least one of the $\mathcal{D}_m$. Therefore we will in the following restrict our attention to the spaces $E_m$ and the operators $\mathcal{D}_m$. The following lemma gives an explicit description of the $\mathcal{D}_m$.
	
	\begin{lemma}\label{lemma_explicit_op}
		Let $\overline{f}\in C^2(-1,1)$ and $h\in \mathcal{H}_m^{n-1}$ for some $m\in\N_0$. Then on $S^{n-1}\setminus\{\pm e_n\}$ we have $\mathcal{D}_m(h\otimes\overline{f})=h\otimes L_m\overline{f}$, where
		\begin{align*}
			L_m\overline{f}(t)=& \  \frac{\eta(t)}{(n-1)\mathcal{A}_2\eta(t)\mathcal{A}_1\eta(t)^{n-2}(1-t^2)^{m+\frac{n-3}{2}}}\frac{d}{dt}\mathcal{A}_1\eta(t)^{n-2}(1-t^2)^{m+\frac{n-1}{2}}\overline{f}'(t)\\
			& \quad + \frac{\eta(t)}{(n-1)\mathcal{A}_2\eta(t)\mathcal{A}_1\eta(t)}k_m(t)\overline{f}(t)
		\end{align*}
		for $t\in(-1,1)$ and
		\begin{align*}
			k_m(t)=-(m-1)(n-2)\mathcal{A}_2\eta(t)-(m^2-1)\mathcal{A}_1\eta(t)-m(m-1)\eta''(t).
		\end{align*}
	\end{lemma}
	
	\begin{remark}
		As $\mathcal{D}_K$ is symmetric on $L^2(S^{n-1},\mu_K)$, $L_m$ must extend to a symmetric operator on the space $L^2\left([-1,1],w_{m,n}(t)dt\right)$ with a suitable domain. This can also be seen directly from the divergence form of $L_m$ given above.
	\end{remark}
	
	\begin{proof}
		By definition, $h$ can be seen as an $m$ homogeneous polynomial on $\R^n$. Denote by $\tilde{h}$ the $0$-homogeneous extension of $h|_{S^{n-1}}$. Denote by $\tilde{f}:\R^n\to\R$ the $1$-homogeneous extension of the function $S^{n-1}\to\R,x\mapsto \overline{f}(x_n)$. Then $f:\R^n\to\R$ given by $f(x)=\tilde{h}(x)\tilde{f}(x)$ is the 1-homogeneous extension of $h\otimes \overline{f}$. Furthermore
		\begin{align*}
			D^2f=& \ \tilde{f}D^2\tilde{h}+ D\tilde{f}(D\tilde{h})^T+D\tilde{h}(D\tilde{f})^T+\tilde{h}D^2\tilde{f}.
		\end{align*}
		By homogeneity for $x\in S^{n-1}$ we have 
		\begin{align*}
			D^2\tilde{h}(x)\big|_{T_xS^{n-1}}=D^2h(x)|_{T_xS^{n-1}}-mh(x)\id.
		\end{align*}
		With $T_xS^{n-1}=X^\perp$ and the assumption that $h$ is harmonic this yields 
		\begin{align}\label{eq_1}
			\tr D^2\tilde{h}\big|_{T_xS^{n-1}}=& \ \tr D^2h(x) -\langle x,D^2h(x)x\rangle - m(n-1)h(x) \nonumber \\
			=& -m(n+m-2)h(x).
		\end{align}
		With $v(x)=e_n-x_nx$ we get
		\begin{align}\label{eq_2}
			\left\langle v(x),D^2\tilde{h}(x)v(x)\right\rangle =& \ x_n^2\langle x,D^2h(x)x\rangle - m(1-x_n^2)h(x) \nonumber \\
			=& \ m(mx_n^2-1)h(x),
		\end{align}
		where we used that $D^2h(x)e_n=0$ as $h$ only depends on the first $(n-1)$ coordinates. 
		
		Furthermore for $x\in S^{n-1}\setminus\{\pm e_n\}$ we have
		\begin{align*}
			D\tilde{f}(x) D\tilde{h}(x)^T= \overline{f}'(x_n)(e_n-x_nx) (Dh(x)-mh(x)x)^T
		\end{align*}
		and therefore
		\begin{align}\label{eq_3}
			\tr \left(D\tilde{f}(x) D\tilde{h}(x)^T\big|_{T_xS^{n-1}}\right) & =\overline{f}'(x_n)\langle e_n-x_nx,Dh(x)-mh(x)x\rangle \nonumber\\
			& = -mx_n\overline{f}'(x_n)h(x)
		\end{align}
		as well as
		\begin{align}\label{eq_4}
			\left\langle v(x), D\tilde{f}(x) D\tilde{h}(x)^Tv(x)\right\rangle = -m(1-x_n^2)x_n\overline{f}'(x_n)h(x).
		\end{align}
		
		By Lemma \ref{zonal_hesssian} 
		\begin{align}\label{eq_5}
			\tr \left(D^2\tilde{f}(x)\big|_{T_xS^{n-1}}\right)=& \  (1-x_n^2)\overline{f}''(x_n)+(n-1)\mathcal{A}_1\overline{f}(x_n)
		\end{align}
		as well as
		\begin{align}\label{eq_6}
			\left\langle v(x),D^2\tilde{f}(x)v(x)\right\rangle =& \ (1-x_n^2)^2\overline{f}''(x_n)+(1-x_n^2)\mathcal{A}_1\overline{f}(x_n).
		\end{align}
		
		Now putting together (\ref{eq_1}), (\ref{eq_3}) and (\ref{eq_5}) we get for $x\in S^{n-1}$
		\begin{align*}
			\tr D^2f(x)|_{T_xS^{n-1}}=& \ h(x)\big(-m(n+m)\overline{f}(x_n) +(1-x_n^2)\overline{f}''(x_n)\\
			& \quad +(n+2m-1)\mathcal{A}_1\overline{f}(x_n)\big).
		\end{align*}
		The identities (\ref{eq_2}), (\ref{eq_4}) and (\ref{eq_6}) readily yield
		\begin{align*}
			\left\langle v(x), D^2f(x)v(x)\right\rangle =& \ h(x)\big(m(mx_n^2-1)\overline{f}(x_n)-2m(1-x_n^2)x_n\overline{f}'(x_n) \\
			& \quad +(1-x_n^2)^2\overline{f}''(x_n)+(1-x_n^2)\mathcal{A}_1\overline{f}(x_n)\big).
		\end{align*}
		Together with Lemma \ref{lemma_explicit_mixed_discriminant} this gives the statement.
	\end{proof}
	
	\section{Proof of the spectral gap}
	
	In order to prove Theorem \ref{spectral_gap_theorem}, we will prove that for all $m\in \N_0$ the operator $\mathcal{D}_m$ has no eigenvalue in the interval $\left(-\frac{1}{n-1},1\right)$ with an eigenfunction that is orthogonal to linear functions. This precisely means that $L_m$ has no eigenvalue in the interval $\left(-\frac{1}{n-1},1\right)$ with an eigenfunction $\overline{f}$ such that $\overline{f}$ is orthogonal to the linear functions on $[-1,1]$ if $m=0$ and orthogonal to constant functions if $m=1$. Accordingly, we will treat the cases $m=0$, $m=1$ and $m\geq 2$ separately. Theorem \ref{theorem_m=0} and Theorem \ref{theorem_m>0} together will complete the proof of Theorem \ref{spectral_gap_theorem}.
	
	\subsection{The operator on $E_0$}\label{section_m=0}
	
	This case will be the foundation of our following work. Essentially in the origin symmetric case, it is already covered by \cite{boeroeczkykalantzopoulos} as it corresponds to the local logarithmic Brunn-Minkowski inequality for two origin symmetric bodies of revolution. Obviously two origin symmetric bodies of revolution are symmetric under reflections along the coordinate hyperplanes and hence the setting of the respective paper applies. As $K$ is assumed to be origin symmetric, the operator $\mathcal{D}_K$ is equivariant under reflection around the origin and hence every eigenfunction of $\mathcal{D}_K$ can be assumed to have a well defined parity. The following theorem is equivalent to the local logarithmic Brunn-Minkowski inequality for two origin symmetric convex bodies of revolution.
	
	\begin{theorem}\label{theorem_eigenvalues_m=0}
		For all even $f\in E_0\cap C^2(S^{n-1})$ with $\langle f,h_K\rangle=0$ we have
		\begin{align*}
			\langle \mathcal{D}_0f,f\rangle+\frac{1}{n-1}\langle f,f\rangle \leq 0.
		\end{align*}
		Stated otherwise, the greatest eigenvalue of $\mathcal{D}_0$ with even eigenfunction is $1$ with eigenfunction $h_K$ and any other eigenvalue with even eigenfunction is at most $-\frac{1}{n-1}$.
	\end{theorem}
	
	This theorem will be proven as a part of Lemma \ref{lemma_m=0_even} in the inequality (\ref{proof_theorem_m=0}).

	However for later purposes and our stronger statement, we need to include also eigenvalues with odd eigenfunctions. To this end, we will use a classical homotopy argument. For $m\in\N_0$ we define the weighted Sobolev space 
	\begin{align*}
		H_m=\left\{\overline{f}\in L^2([-1,1],(1-t^2)^{m+\frac{n-3}{2}}dt):\overline{f}'\in L^2\left([-1,1],(1-t^2)^{m+\frac{n-1}{2}}dt\right)\right\}
	\end{align*}
	equipped with the inner product
	\begin{align*}
		\langle \overline{f}_1,\overline{f}_2\rangle_{H_m}=\int_{-1}^1 \left(\overline{f}_1(t) \overline{f}_2(t) (1-t^2)^{m+\frac{n-3}{2}}+\overline{f}_1'(t)\overline{f}_2'(t)(1-t^2)^{m+\frac{n-1}{2}}\right)dt
	\end{align*}
	and the induced norm $\norm{\cdot}_{H_m}$. Note, that the $0$-order term of the Sobolev-inner product given by
	\begin{align*}
		(\overline{f}_1,\overline{f}_2)\mapsto \int_{-1}^1 \overline{f}_1(t) \overline{f}_2(t) (1-t^2)^{m+\frac{n-3}{2}}dt
	\end{align*}
	induces the same topology as $\langle\cdot,\cdot\rangle_{m,n}$, as the weights only differ by multiplication with a function that is uniformly bounded from above and from below by some positive constant. In this section, we focus on the operator $\mathcal{D}_0$ on $E_0$. For this we only need the space $H_0$.
	
	\begin{theorem}\label{theorem_m=0}
		For all $\overline{f}\in H_0$ with $\langle \overline{f},\eta\rangle_{0,n}=\langle \overline{f},t\rangle_{0,n}=0$ and Neumann boundary conditions
		\begin{align}\label{boundary_conditions_m=0}
			(1-t^2)^{\frac{n-1}{2}}\overline{f}'(t)u(t)\big|_{t=-1}^1=0 \quad \text{for all } u\in H_0
		\end{align}
		we have
		\begin{align*}
			\langle L_0\overline{f},\overline{f}\rangle_{0,n} + \frac{1}{n-1}\langle \overline{f},\overline{f}\rangle_{0,n}\leq 0.
		\end{align*}
		As every $f\in E_0\cap C^2(S^{n-1})$ is given by $f(x)=\overline{f}(x_n)$ for some $\overline{f}\in C^2[-1,1]\subset H_0$ satisfying (\ref{boundary_conditions_m=0}), this proves that any eigenvalue $\lambda\notin\{0,1\}$ of $\mathcal{D}_0$ satisfies
		\begin{align*}
			\lambda\leq -\frac{1}{n-1}.
		\end{align*}
	\end{theorem}
	
	\begin{remark}
		Here, the boundary conditions (\ref{boundary_conditions_m=0}) are to be understood in the weak sense
		\begin{align}\label{weak_formulation}
			&\langle L_0\overline{f},\overline{f}\rangle_{0,n} \nonumber \\
			& \quad =\frac{1}{n(n-1)}\int_{-1}^1 \left(-\mathcal{A}_1\eta(t)^{n-2}(1-t^2)^{\frac{n-1}{2}}\overline{f}'(t)^2+ \mathcal{A}_1\eta(t)^{n-3}(1-t^2)^{\frac{n-3}{2}}k_0(t)\overline{f}(t)^2\right)dt.
		\end{align}
	\end{remark}
	
	For the proof of Theorem \ref{theorem_m=0} we need the following information about $H_0$. As we need all the technical results later for the spaces $H_m$ with $m>0$, we state everything for arbitrary $m\in\N_0$ at once.
	
	\begin{lemma}\label{lemma_compact_embedding}
		For any $m\in\N_0$ we have the following.
		\begin{enumerate}
			\item $C^2[-1,1]$ is dense in $H_m$.
			\item If $m\geq 1$ then for every $\overline{f}\in H_m$ we have $\lim\limits_{t\to\pm 1}(1-t^2)^{m+\frac{n-3}{2}}\overline{f}(t)^2=0$. 
			\item $H_m$ embeds compactly in $L^2([-1,1],w_{m,n}(t)dt)$. 
		\end{enumerate} 
	\end{lemma}
	\begin{proof}
		For (i) assume $\overline{f}\in H_m$. Then by definition the mapping $t\to (1-t^2)^{\frac{m}{2}+\frac{n-1}{4}}\overline{f}'(t)$ is in $L^2[-1,1]$ and hence it can be approximated by a sequence $\overline{g}_k\in C^1_c(-1,1)$. For $\varepsilon>0$ the weights $(1-t^2)^{m+\frac{n-1}{2}}$ and $(1-t^2)^{m+\frac{n-3}{2}}$ are uniformly bounded on $[-1+\varepsilon,1-\varepsilon]$ from above and from below by some positive constant. Therefore $\overline{f}|_{[-1+\varepsilon,1-\varepsilon]}\in H^1[-1+\varepsilon,1-\varepsilon]$. By the Sobolev embedding theorem, $\overline{f}\in C(-1,1)$. Then an easy calculation shows that the sequence $(\overline{f}_k)_{k\in\N}\subset C^2[-1,1]$ given by
		\begin{align*}
			\overline{f}_k(t)=\int_0^t (1-s^2)^{-\frac{m}{2}-\frac{n-1}{4}}\overline{g}_k(s)ds + \overline{f}(0)
		\end{align*}
		converges to $\overline{f}$ in $H_m$.
		
		For (ii) let $\varepsilon>0$. Then we have
		\begin{align*}
			\overline{f}(t)=\int_{1-\varepsilon}^{t}\overline{f}'(s)ds + \overline{f}(1-\varepsilon)
		\end{align*}
		and hence by the Cauchy-Schwarz inequality
		\begin{align*}
			&(1-t^2)^{m+\frac{n-3}{2}}\overline{f}(t)^2 \\
			& \quad \leq 2(1-t^2)^{m+\frac{n-3}{2}}\left(\int_{1-\varepsilon}^t \overline{f}'(s)ds\right)^2 + 2(1-t^2)^{m+\frac{n-3}{2}}\overline{f}(1-\varepsilon)^2 \\
			& \quad \leq 2(1-t^2)^{m+\frac{n-3}{2}}\left(\int_{1-\varepsilon}^t (1-s^2)^{-m-\frac{n-1}{2}}ds\right)\left(\int_{1-\varepsilon}^t (1-s^2)^{m+\frac{n-1}{2}}\overline{f}'(s)^2ds\right) \\
			& \hspace{5cm} +  2(1-t^2)^{m+\frac{n-3}{2}}\overline{f}(1-\varepsilon)^2.
		\end{align*}
		By l'Hospital's rule, the constant $C_{m,n}=\lim\limits_{t\to 1} (1-t^2)^{m+\frac{n-3}{2}}\left(\int_{-t}^t (1-s^2)^{-m-\frac{n-1}{2}}ds\right)$ is well defined, hence using $m\geq 1$ we have
		\begin{align*}
			\limsup\limits_{t\to 1} \ (1-t^2)^{m+\frac{n-3}{2}}\overline{f}(t)^2 \leq C_{m,n}\left(\int_{1-\varepsilon}^1 (1-s^2)^{m+\frac{n-1}{2}}\overline{f}'(s)^2ds\right).
		\end{align*}
		Now $\varepsilon\to 0$ proves the statement for $t\to 1$. The case $t\to -1$ is shown analogously.
		
		For (iii) let $(\overline{f}_k)_{k\in\N}$ be a bounded sequence in $H_m$. As above then $\overline{f}_k|_{[-1+\varepsilon,1-\varepsilon]}\in H^1[-1+\varepsilon,1-\varepsilon]$ for all $k\in\N$. By the Rellich-Kondrachov theorem, $H^1[-1+\varepsilon,1-\varepsilon]$ embeds compactly into $L^2[-1+\varepsilon,1-\varepsilon]$ and hence we can construct successive subsequences that are convergent in the space
		\begin{align*}
			L^2\left(\left[-1+\frac{1}{N},1-\frac{1}{N}\right],w_{m,n}(t)dt\right)
		\end{align*}
		for $N>1$. We claim that $(w_{m,n}\overline{f}_k^2)_{k\in\N}$ is uniformly bounded on $(-1,1)$. Then using a standard diagonal sequence argument we can construct a subsequence that converges in $L^2([-1,1],w_{m,n}(t)dt)$. To prove that $(w_{m,n}\overline{f}_k^2)_{k\in\N}$ is uniformly bounded on $(-1,1)$ observe that for any $\overline{f}\in H_m$ we have as for (ii)
		\begin{align*}
			\overline{f}(t)^2\leq C_{m,n} \left(\int_{-1}^1 (1-s^2)^{m+\frac{n-1}{2}}\overline{f}'(s)^2ds\right)+2\overline{f}(0)^2
		\end{align*}
		for $t\in(-1,1)$. By the Sobolev embedding theorem the right hand side is uniformly bounded in terms of $\norm{\overline{f}}_{H_m}$. This proves the claim and thereby the compactness of the embedding. 
	\end{proof}
	
	First, we prove Theorem \ref{theorem_m=0} for even $\overline{f}$ generalizing Theorem \ref{theorem_eigenvalues_m=0}. This uses the known local logarithmic Brunn-Minkowski inequality for two origin symmetric convex bodies of revolution.
	
	\begin{lemma}\label{lemma_m=0_even}
		For any even $\overline{f}\in H_0$ with $\langle \overline{f},\eta\rangle_{0,n}=0$ and Neumann boundary conditions (\ref{boundary_conditions_m=0}) we have
		\begin{align}\label{ineq_m=0_even}
			\langle L_0\overline{f},\overline{f}\rangle_{0,n} + \frac{1}{n-1}\langle \overline{f},\overline{f}\rangle_{0,n}\leq 0.
		\end{align}
	\end{lemma}
	\begin{proof}[Proof of Theorem \ref{theorem_eigenvalues_m=0} and Lemma \ref{lemma_m=0_even}]
		First assume $\overline{f}\in C^2[-1,1]$ and let $f\in C^2(S^{n-1})$ be defined by $f(x)=\overline{f}(x_n)$. Then for $c>0$ sufficiently large, $f+ch_K$ is the support function of a smooth convex body $L$ and hence $f=h_L-ch_K$. Furthermore by definition, $f$ is zonal and hence $L$ is a body of revolution. The condition $\langle \overline{f},\eta\rangle_{0,n}=0$ is equivalent to $\langle f,h_k\rangle=0$ and implies $f=h_L-\frac{\langle h_L,h_K\rangle}{\langle h_K,h_K\rangle}h_K$. As any pair of two origin symmetric bodies of revolution is symmetric under reflections along the coordinate hyperplanes, \cite[Theorem 2]{boeroeczkykalantzopoulos} shows that the pair $(K,L)$ satisfies (\ref{log_bm_ineq}). By the classical global to local implication (c.f. \cite[Theorem 3.4]{kolesnikovmilman} adapted to the setting of symmetric bodies), this implies (\ref{loc_log_bm_ineq}). Then the same calculation as in the proof of Theorem \ref{local_log_bm_zonal} shows that
		\begin{align}\label{proof_theorem_m=0}
			\langle \mathcal{D}_0 f,f\rangle + \frac{1}{n-1}\langle f,f\rangle \leq 0.
		\end{align}
		This proves Theorem \ref{theorem_eigenvalues_m=0}. By Lemma \ref{lemma_explicit_op}, the inequality (\ref{proof_theorem_m=0}) precisely means that $\overline{f}$ satisfies (\ref{ineq_m=0_even}).
		As the left hand side of \eqref{ineq_m=0_even} is clearly continuous in $H_0$, by Lemma \ref{lemma_compact_embedding}(i) this extends to all $\overline{f}\in H_0$ proving the statement.
		
	\end{proof}
	
	Without the assumption that $f$ is even, we cannot reduce our problem to a known case of the local logarithmic Brunn-Minkowski inequality. We therefore need to closer investigate the eigenvalue problem. For the analysis of the operator we need the following theorem, which is the baseline of what is often referred to as the Frobenius method. It studies linear differential equations with singular coefficients of the form
	\begin{align}\label{fuchs_eq}
		u''(z)+\frac{p(z)}{z-z_0}u'(z)+\frac{q(z)}{(z-z_0)^2}u(z)=0,
	\end{align}
	where $p$ and $q$ are holomorphic functions on some domain that contains $z_0$.
	
	\begin{theorem}\label{frobenius_method}
		Fix $z_0\in\C$ and let $p(z)=\sum_{k=0}^\infty p_k(z-z_0)^k$ as well as $q(z)=\sum_{k=0}^\infty q_k(z-z_0)^k$ be holomorphic functions on $B_R(z_0)$ for $R>0$. Denote by $\alpha_1$, $\alpha_2$ the roots of the polynomial
		\begin{align}\label{indical_polynomial}
			P(\alpha)=\alpha^2+(p_0-1)\alpha+q_0,
		\end{align}
		where $\Rea(\alpha_1)\geq \Rea(\alpha_2)$. Then there exists a unique holomorphic function $h_1:B_R(z_0)\to\C$ with $h_1(z_0)=1$, such that the function $u_1:B_R(z_0)\setminus (z_0+(-\infty,0])\to\C$ given by $u_1(z)=(z-z_0)^{\alpha_1} h_1(z)$ solves (\ref{fuchs_eq}). 
		
		If $\alpha_1-\alpha_2\notin \N_0$ then there exists a solution $u_2:B_R(z_0)\setminus (z_0+(-\infty,0])\to\C$ to (\ref{fuchs_eq}) given by $u_2(z)=(z-z_0)^{\alpha_2} h_2(z)$ for some holomorphic function $h_2:B_R(z_0)\to\C$ with $h_2(z_0)=1$. If on the other hand $\alpha_1-\alpha_2\in\N_0$ then there exists $c\in\C$ and a holomorphic function $h_2:B_R(z_0)\to\C$ with $h_2(z_0)=1$, such that $u_2:B_R(z_0)\setminus (z_0+(-\infty,0])\to\C$ to (\ref{fuchs_eq}) given by $u_2(z)=(z-z_0)^{\alpha_2} h_2(z)+c(z-z_0)^{\alpha_1}h_1(z)\log(z-z_0)$ is a solution to (\ref{fuchs_eq}) that is linearly independent to $u_1$.
	\end{theorem}
	
	The existence of solutions to the equation (\ref{fuchs_eq}) of the given form was first proven by Fuchs, whereas the work of Frobenius majorly consisted of analyzing the domain of the solutions. In contribution to that, the method is called the Frobenius method and the solutions $\alpha_{1/2}$ to the polynomial (\ref{indical_polynomial}) are called the Frobenius indices. The proof of the theorem can be found in many standard textbooks, see e. g. \cite[Section V]{teschl}.
	
	To apply the theorem, we need to assume that $\eta$ is analytic, which is however no loss of generality, as every $C^2$ function can be approximated by analytic functions in the $C^2$ sense. 
	
	\begin{lemma}\label{lemma_sigle_ev}
		Assume that $\eta$ is analytic and $m\in\N_0$. Then for any $\lambda\in \R$ there is at most a one dimensional family of weak solutions in $H_m$ to the equation
		\begin{align}\label{sturm_liouville}
			L_m\overline{f}=\lambda\overline{f}
		\end{align}
		with Neumann boundary conditions
		\begin{align}\label{boundary_conditions_arbitrary_m}
			(1-t^2)^{m+\frac{n-1}{2}}\overline{f}'(t)u(t)\big|_{t=-1}^1=0 \quad \text{for all } u\in H_m.
		\end{align}
	\end{lemma}
	
	\begin{remark}\label{remark_weak_solution}
		As before, the problem is to be understood in a weak sense, i.e. a solution $\overline{f}\in H_m$ satisfies
		\begin{align*}
			0 & = -\frac{1}{n-1}\int_{-1}^1 \mathcal{A}_1\eta(t)^{n-2}(1-t^2)^{m+\frac{n-1}{2}}\overline{f}'(t)u'(t)dt \\
			& \quad + \int_{-1}^1 \frac{\mathcal{A}_1\eta(t)^{n-3}(1-t^2)^{m+\frac{n-3}{2}}}{n-1}\left(k_m(t)-\frac{\lambda \mathcal{A}_2\eta(t)\mathcal{A}_1\eta(t)}{\eta(t)}\right)\overline{f}(t)u(t)dt
		\end{align*} 
		for all $u\in H_m$. However, standard regularity theory yields that for analytic $\eta$ any weak solution $\overline{f}$ is in fact in $C^2(-1,1)$, and hence the weak formulation is equivalent to the classical problem with boundary conditions in the classical sense by simple integration by parts. Furthermore we note, that if in (\ref{boundary_conditions_m=0}) we choose $u$ to be continuous on $[-1,1]$ and concentrated around one endpoint, this implies 
		\begin{align}\label{strong_Neumann_bc}
			\lim\limits_{t\to\pm 1} (1-t^2)^{m+\frac{n-1}{2}}\overline{f}'(t)=0.
		\end{align}
	\end{remark}
	
	\begin{proof}
		Rewriting (\ref{sturm_liouville}) in the form of (\ref{fuchs_eq}) yields 
		\begin{align*}
			&\overline{f}''(t)+\frac{1}{1-t^2}\left(-(n-2)\frac{\mathcal{A}_2\eta(t)}{\mathcal{A}_1\eta(t)}-(2m+1)\right)t \overline{f}'(t) \\
			&\quad +\frac{1}{1-t^2}\left((n-2)\frac{\mathcal{A}_2\eta(t)}{\mathcal{A}_1\eta(t)}+2m+1+\frac{(n-1)\lambda\mathcal{A}_2\eta(t)}{\eta(t)}\right)\overline{f}(t)=0
		\end{align*}
		or equivalently 
		\begin{align*}
			\overline{f}''(t)+\frac{p^\pm(t)}{t \mp 1}\overline{f}'(t)+\frac{q^\pm(t)}{(t \mp 1)^2}\overline{f}(t)=0
		\end{align*}
		with 
		\begin{align*}
			p^\pm(t) & =\frac{t}{t \pm 1}((n-2) \frac{\mathcal A_2(t)}{\mathcal A_1(t)}+2m+1)\\
			q^\pm(t) & = \frac{1 \mp t}{1 \pm t} \left((n-2)\frac{\mathcal{A}_2\eta(t)}{\mathcal{A}_1\eta(t)}+2m+1+\frac{(n-1)\lambda\mathcal{A}_2\eta(t)}{\eta(t)}\right).
		\end{align*}
		Now as remarked before, any solution in $H_m$ is even $C^2(-1,1)$ and hence the existence and uniqueness theorem implies, that it is a linear combination of the two solutions given by the Frobenius method. As $p^\pm(\pm 1)=m+\frac{n-1}{2}$ and $q^{\pm}(\pm 1)=0$ the Frobenius indices satisfy 
		\begin{align*}
			\alpha^2+ \left(m+\frac{n-3}{2}\right)\alpha=0
		\end{align*}
		leading to $\alpha_1=0$ and $\alpha_2=-m-\frac{n-3}{2}$. If $n\neq 3$ or $m>0$ then at least one solution grows like $(t- 1)^{-m-\frac{n-3}{2}}$ around $1$. Therefore its derivative behaves as $(t-1)^{-m-\frac{n-1}{2}}$ for $t\to1$ and does not satisfy (\ref{strong_Neumann_bc}). Therefore by Remark \ref{remark_weak_solution} it does not satisfy the Neumann boundary conditions (\ref{boundary_conditions_arbitrary_m}). For $n=3$ and $m=0$ both indices are the same and therefore by the uniqueness statement in Theorem \ref{frobenius_method}, at least one solution must contain a logarithmic term around $1$. Again, this means that the respective derivative does not satisfy the boundary conditions.
	\end{proof}
	
	We can now use a classical homotopy argument to prove Theorem \ref{theorem_m=0}. 
	
	\begin{proof}[Proof of Theorem \ref{theorem_m=0}]
		Without loss of generality we can assume that $\eta$ is analytic. Now for $s\in [0,1]$ define $\eta_s(t)=1-s+s\eta(t)$ and let $L_{0,s}$ be the operator $L_0$ defined by $\eta_s$ as well as $\langle\cdot,\cdot\rangle_{0,n,s}$ the respective inner product. Assume boundary conditions (\ref{boundary_conditions_m=0}) and consider the function
		\begin{align}\label{definition_F(s)}
			F(s)= \sup_{\substack{0\neq\overline{f}\in H_0 \text{ even}\\ \overline{f}\perp \eta_s}} \frac{\langle L_{0,s}\overline{f},\overline{f}\rangle_{0,n,s}}{\langle \overline{f},\overline{f}\rangle_{0,n,s}} - \sup_{\substack{0\neq\overline{f}\in H_0 \text{ odd}\\ \overline{f}\perp t}} \frac{\langle L_{0,s}\overline{f},\overline{f}\rangle_{0,n,s}}{\langle \overline{f},\overline{f}\rangle_{0,n,s}}.
		\end{align}
		Now Lemma \ref{lemma_m=0_even} states that the first summand is $\leq -\frac{1}{n-1}$ for all $s$. What we want to prove is, that for $s=1$, both suprema are $\leq -\frac{1}{n-1}$. Assume that this would not be the case, then this means $F(1)<0$. 
		
		We first investigate the case $s=0$ and claim that $F(0)>0$. With $\eta_0=1$ we have 
		\begin{align*}
			\langle \overline{f}_1,\overline{f}_2\rangle_{0,n,0}=\frac{1}{n}\int_{-1}^1 (1-t^2)^{\frac{n-3}{2}}\overline{f}_1(t)\overline{f}_2(t)dt
		\end{align*}
		as well as
		\begin{align*}
			L_{0,0}\overline{f}(t)=\frac{1}{n-1} \cdot  \frac{1}{(1-t^2)^{\frac{n-3}{2}}}\frac{d}{dt}(1-t^2)^{\frac{n-1}{2}}\overline{f}'(t)+\overline{f}(t).
		\end{align*}
		
		We state that for $s=0$ each of the two quotients in (\ref{definition_F(s)}) has a maximizer in $H_0$. This can be seen by a classical compactness argument as follows. Assume, that $(\overline{f}_k)_{k\in\N}\subset H_0$ is a maximizing sequence satisfying $\langle \overline{f}_k,\overline{f}_k\rangle_{0,n,0}=1$. Then we have
		\begin{align*}
			\langle L_0\overline{f}_k,\overline{f}_k\rangle_{0,n} =& \ \frac{1}{n}\int_{-1}^1 \left(-\frac{1}{n-1}(1-t^2)^{\frac{n-1}{2}}\overline{f}_k'(t)^2 + (1-t^2)^{\frac{n-3}{2}}\overline{f}_k(t)^2\right)dt  \\
			=& \ 1 -\frac{1}{n}\int_{-1}^1 \frac{1}{n-1}(1-t^2)^{\frac{n-1}{2}}\overline{f}_k'(t)^2 dt.
		\end{align*}
		Consequently the sequence $(\overline{f}_k)_{k\in\N}$ is bounded in $H_0$ and hence there is a weakly convergent subsequence. Assume without loss of generality $\overline{f}_k\rightharpoonup \overline{f}$ with $\overline{f}\in H_0$. By Lemma \ref{lemma_compact_embedding}(iii) then $\overline{f}_k\to\overline{f}$ in $L^2([-1,1],w_{0,n}(t)dt)$. By continuity $\overline{f}$ obviously has the same parity and orthogonality properties as the $\overline{f}_k$. By the lower semi-continuity of the norm with respect to weak convergence
		\begin{align*}
			\limsup\limits_{k\to\infty} \langle L_{0,0}\overline{f}_k,\overline{f}_k\rangle_{0,n,0} \leq \langle L_{0,0}\overline{f},\overline{f}\rangle_{0,n,0}
		\end{align*}
		and hence $\overline{f}$ is indeed a maximizer.
		
		By the standard theory of variation such a maximizer is a weak solution to
		\begin{align*}
			L_{0,0}\overline{f}=\mu \overline{f}
		\end{align*}
		with boundary conditions (\ref{boundary_conditions_m=0}). Those solutions are well known to be the Gegenbauer polynomials $C^{\left(\frac{n-2}{2}\right)}_k$ for $k\in\N_0$ where $\mu_k=1-\frac{k(k+n-2)}{n-1}$ is the respective eigenvalue. Each $C^{\left(\frac{n-2}{2}\right)}_k$ is a polynomial of degree $k$. If $k$ is even, then $C^{\left(\frac{n-2}{2}\right)}_k$ is even and if $k$ is odd then $C^{\left(\frac{n-2}{2}\right)}_k$ is odd. As $C^{\left(\frac{n-2}{2}\right)}_0(t)=1$ and $C^{\left(\frac{n-2}{2}\right)}_1(t)=(n-2)t$, the maximizers in (\ref{definition_F(s)}) are given by $C^{\left(\frac{n-2}{2}\right)}_2$ as well as $C^{\left(\frac{n-2}{2}\right)}_3$ and therefore
		\begin{align*}
			F(0)= \mu_2 - \mu_3 > 0.
		\end{align*}
		Consequently, the continuity of $F$ implies that there is some $s\in(0,1)$ for which $F(s)=0$. As $\mathcal{A}_2\eta$, $\mathcal{A}_1\eta$ and $\eta$ are all bounded from above and from below by some positive constant, the same argument as for the $s=0$ case implies, that there exist $\overline{f}_{\text{even}}\in H_0$ even and $\overline{f}_{\text{odd}}\in H_0$ odd such that
		\begin{align*}
			L_{0,s}\overline{f}_{\text{even}}=\mu_{\text{even}}\overline{f}_{\text{even}} \quad \text{ as well as} \quad L_{0,s}\overline{f}_{\text{odd}}=\mu_{\text{odd}}\overline{f}_{\text{odd}}
		\end{align*}
		with boundary conditions (\ref{boundary_conditions_m=0}) where
		\begin{align*}
			\mu_{\text{even}}=\sup_{\substack{0\neq\overline{f}\in H_0 \text{ even}\\ \overline{f}\perp \eta_s}} \frac{\langle L_{0,s}\overline{f},\overline{f}\rangle_{0,n,s}}{\langle \overline{f},\overline{f}\rangle_{0,n,s}} \quad \text{ as well as} \quad \mu_{\text{odd}} = \sup_{\substack{0\neq\overline{f}\in H_0 \text{ odd}\\ \overline{f}\perp t}} \frac{\langle L_{0,s}\overline{f},\overline{f}\rangle_{0,n,s}}{\langle \overline{f},\overline{f}\rangle_{0,n,s}}.
		\end{align*}
		On the other hand, by definition, $F(s)=0$ then implies $\mu_{\text{even}}=\mu_{\text{odd}}$ and hence $\overline{f}_{\text{even}}$ and $\overline{f}_{\text{odd}}$ are two linearly independent solutions to (\ref{sturm_liouville}) for the same $\lambda$ in contradiction to Lemma \ref{lemma_sigle_ev}.	
	\end{proof}

	We can now deduce the behavior of $\mathcal{D}_m$ on $E_m$ for $m\geq1$ from the proven bounds on $\mathcal{D}_0$. To this end, we treat the cases $m=1$ and $m\geq 2$ separately. 
	
	\subsection{The operator on $E_1$}\label{section_m=1}
	
	Our aim is to prove the following theorem.
	
	\begin{theorem}\label{theorem_m=1}
		For all $\overline{f}\in H_1$ with boundary conditions (\ref{boundary_conditions_arbitrary_m}) for $m=1$ and $\langle \overline{f},1\rangle_{1,n}=0$ we have
		\begin{align}\label{inner_product_ineq_m=1}
			\langle L_1\overline{f},\overline{f}\rangle_{1,n}+\frac{1}{n-1}\langle \overline{f},\overline{f}\rangle_{1,n}\leq 0.
		\end{align}
	\end{theorem}
	
	We first prove some identity that allows us to reduce the behavior of the operator $L_m$ on $H_m$ to the behavior of the operator $L_0$ on $H_0$. We will use this reduction to prove Lemma \ref{odd_case} as well as Lemma \ref{even_case} and by means of that Theorem \ref{theorem_m=1} as well as Theorem \ref{theorem_m>1} respectively.
	
	\begin{lemma}\label{lemma_orthogonal_space}
		Suppose $\overline{f} \in H_m$ satisfies Neumann boundary conditions (\ref{boundary_conditions_arbitrary_m}). Define $\overline{g}:[-1,1]\to\R$ by $\overline{g}(t)=(1-t^2)^{\frac{m}{2}}\overline{f}(t)$. Then $\overline{g}\in H_0$ and satisfies the Neumann boundary conditions (\ref{boundary_conditions_m=0}). Furthermore
		\begin{align*}
			&\langle L_m \overline{f},\overline{f}\rangle_{m,n}+\frac{1}{n-1}\langle \overline{f},\overline{f}\rangle_{m,n} \\
			& \quad = \langle L_0\overline{g},\overline{g}\rangle_{0,n} + \frac{1}{n-1}\langle \overline{g},\overline{g}\rangle_{0,n} \\
			& \quad \quad -\frac{m(m+n-3)}{n(n-1)}\int_{-1}^1 \mathcal{A}_2\eta(t)\mathcal{A}_1\eta(t)^{n-3}(1-t^2)^{m+\frac{n-5}{2}} \overline{f}(t)^2dt.
		\end{align*}
	\end{lemma}
	
	\begin{proof}
		Clearly $\overline{g}\in L^2([-1,1],(1-t^2)^{\frac{n-3}{2}}dt)$ by definition. Integration by parts and using 
		\begin{align*}
			&\frac{d}{dt}\mathcal{A}_1\eta(t)^{n-2}(1-t^2)^{m+\frac{n-3}{2}}\\
			& \quad =-t\mathcal{A}_1\eta(t)^{n-3}(1-t^2)^{m+\frac{n-5}{2}}\left((n-2)\mathcal{A}_2\eta(t)+(2m-1)\mathcal{A}_1\eta(t)\right)
		\end{align*}
		yields 	
		\begin{align*}
			&\int_{-1} ^1 \mathcal{A}_1\eta(t)^{n-2}(1-t^2)^{\frac{n-1}{2}}\overline{g}'(t)^2dt \\
			&\quad = \int_{-1} ^1 \mathcal{A}_1\eta(t)^{n-2}(1-t^2)^{m+\frac{n-5}{2}}\left((1-t^2)^2\overline{f}'(t)^2-2mt(1-t^2)\overline{f}'(t)\overline{f}(t)+m^2t^2\overline{f}(t)^2\right)dt\\
			& \quad = \int_{-1} ^1 \mathcal{A}_1\eta(t)^{n-2}(1-t^2)^{m+\frac{n-5}{2}}\bigg((1-t^2)^2\overline{f}'(t)^2-m(1-t^2)\frac{d}{dt}t\overline{f}(t)^2\\
			& \hspace{9cm} +\left(m^2t^2+m(1-t^2)\right)\overline{f}(t)^2\bigg)dt \\
			& \quad = \int_{-1}^1 \mathcal{A}_1\eta(t)^{n-2}(1-t^2)^{m+\frac{n-1}{2}}\overline{f}'(t)^2 dt \\
			& \quad \quad -m\int_{-1}^1 \mathcal{A}_1\eta(t)^{n-3} (1-t^2)^{m+\frac{n-5}{2}} t^2 \left((n-2)\mathcal{A}_2\eta(t)+(m-1)\mathcal{A}_1\eta(t)\right) \overline{f}(t)^2dt \\
			& \quad \quad + m \int_{-1} ^1\mathcal{A}_1\eta(t)^{n-2}(1-t^2)^{m+\frac{n-3}{2}}\overline{f}(t)^2dt.
		\end{align*} 
		Here we used that $(1-t^2)^{m+\frac{n-3}{2}}\overline{f}(t)^2\big|_{t=-1}^1=0$ by Lemma \ref{lemma_compact_embedding}(ii) and hence the boundary terms from the integration by parts vanish. Especially this shows $\overline{g}\in H_0$.
		
		For the boundary conditions observe that for any $u\in H_0\subset H_m$ by the boundary conditions on $\overline{f}$ and Lemma \ref{lemma_compact_embedding}(ii) we have
		\begin{align*}
			(1-t^2)^{\frac{n-1}{2}}\overline{g}(t)u(t)\big|_{t=-1}^1=& \ (1-t^2)^{\frac{m+n-1}{2}}\overline{f}'(t)u(t)\big|_{t=-1}^1 - mt(1-t^2)^{\frac{m+n-3}{2}}f(t)u(t)\big|_{t=-1}^1 \\
			=& \ -mt \left((1-t^2)^{m+\frac{n-3}{2}}\overline{f}(t)^2(1-t^2)^{\frac{n-3}{2}}u(t)^2\right)^{\frac{1}{2}}\big|_{t=-1}^1 \\
			=& \ 0.
		\end{align*}
		
		With the above, a short calculation shows that
		\begin{align*}
			& n(n-1)\langle L_0\overline{g},\overline{g}\rangle_{0,n} + n\langle \overline{g},\overline{g}\rangle_{0,n} \\
			& \quad = -\int_{-1} ^1 \mathcal{A}_1\eta(t)^{n-2}(1-t^2)^{\frac{n-1}{2}}\overline{g}'(t)^2 dt \\
			& \hspace{2cm} + \int_{-1}^1 \mathcal{A}_1\eta(t)^{n-3}(1-t^2)^{m+\frac{n-3}{2}}\left(k_0(t)+\frac{\mathcal{A}_2\eta(t)\mathcal{A}_1\eta(t)}{\eta(t)}\right)\overline{f}(t)^2 dt \\
%			=& -\int_{-1}^1 \mathcal{A}_1\eta(t)^{n-2}(1-t^2)^{m+\frac{n-1}{2}}\overline{f}'(t)^2 dt \\
%			& \quad +m\int_{-1}^1 \mathcal{A}_1\eta(t)^{n-3} (1-t^2)^{m+\frac{n-5}{2}} t^2 \left((n-2)\mathcal{A}_2\eta(t)+(m-1)\mathcal{A}_1\eta(t)\right) \overline{f}(t)^2dt \\
%			& \quad - m \int_{-1} ^1\mathcal{A}_1\eta(t)^{n-2}(1-t^2)^{m+\frac{n-3}{2}}\overline{f}(t)^2dt \\
%			&\quad + \int_{-1}^1 \mathcal{A}_1\eta(t)^{n-3}(1-t^2)^{m+\frac{n-3}{2}}\left((n-2)\mathcal{A}_2\eta(t)+\mathcal{A}_1\eta(t)+\frac{\mathcal{A}_2\eta(t)\mathcal{A}_1\eta(t)}{\eta(t)}\right)\overline{f}(t)^2 dt \\
%			=& -\int_{-1}^1 \mathcal{A}_1\eta(t)^{n-2}(1-t^2)^{m+\frac{n-1}{2}}\overline{f}'(t)^2 dt \\
%			& \quad - \int_{-1}^1 \mathcal{A}_1\eta(t)^{n-3}(1-t^2)^{m+\frac{n-3}{2}}\left((m-1)(n-2)\mathcal{A}_2\eta(t)+(m^2-1)\mathcal{A}_1\eta(t)\right)\overline{f}(t)^2 dt \\
%			& \quad + \int_{-1}^1\mathcal{A}_1\eta(t)^{n-3}(1-t^2)^{m+\frac{n-3}{2}} \left(-m(m-1)\eta''(t)+\frac{\mathcal{A}_2\eta(t)\mathcal{A}_1\eta(t)}{\eta(t)}\right)\overline{f}(t)^2dt\\
%			& \quad + m(m+n-3)\int_{-1}^1 \mathcal{A}_2\eta(t)\mathcal{A}_1\eta(t)^{n-3}(1-t^2)^{m+\frac{n-5}{2}}\overline{f}(t)^2 dt \\
			& \quad = n(n-1)\langle L_m\overline{f},\overline{f}\rangle_{m,n}  + n\langle \overline{f},\overline{f}\rangle_{m,n}\\
			& \hspace{2cm} + m(m+n-3)\int_{-1}^1 \mathcal{A}_2\eta(t)\mathcal{A}_1\eta(t)^{n-3}(1-t^2)^{m+\frac{n-5}{2}}\overline{f}(t)^2dt.
		\end{align*}
		This proves the statement.
	\end{proof}

	The following two Lemmas prove Theorem \ref{theorem_m=1} for eigenvalues with odd eigenfunctions. We will then use a homotopy argument as in Section \ref{section_m=0} to prove Theorem \ref{theorem_m=1}.
	
	\begin{lemma}\label{odd_case}
		If $\overline{f} \in H_1$ is odd and satisfies (\ref{boundary_conditions_arbitrary_m}) for $m=1$ then
		\begin{align}\label{ineq_odd_case}
			& \langle L_1\overline{f},\overline{f}\rangle_{1,n} + \frac{1}{n-1}\langle \overline{f},\overline{f}\rangle_{1,n} \nonumber \\
			& \quad \leq \frac{1}{n(n-1)}\frac{\left(\int_{-1}^1\frac{\mathcal{A}_2\eta(t)\mathcal{A}_1\eta(t)^{n-2}}{\eta(t)}(1-t^2)^{\frac{n-2}{2}}t\overline{f}(t)dt\right)^2}{\int_{-1}^1\frac{\mathcal{A}_2\eta(t)\mathcal{A}_1\eta(t)^{n-2}}{\eta(t)}(1-t^2)^{\frac{n-3}{2}}t^2dt} \\
			& \quad \quad  -\frac{n-2}{n(n-1)} \int_{-1}^1 \mathcal{A}_2\eta(t)\mathcal{A}_1\eta(t)^{n-3} (1-t^2)^{\frac{n-3}{2}} \overline{f}(t)^2dt. \nonumber
		\end{align}
	\end{lemma}
	
	\begin{proof}
		As (\ref{ineq_odd_case}) is invariant under scaling of $\overline{f}$ we can assume without loss of generality $\overline{f}(t)=(1-t^2)^{-\frac{1}{2}}\left(t+\overline{g}^\perp(t)\right)$ with $\langle \overline{g}^\perp,t\rangle_{0,n}=0$.
		
		By orthogonality, we have
		\begin{align*}
			\langle t+\overline{g}^\perp,t+\overline{g}\rangle_{0,n}=& \langle t,t\rangle_{0,n} + \langle \overline{g}^\perp,\overline{g}^\perp\rangle_{0,n}.
		\end{align*}
		This together with a short calculation using Lemma \ref{lemma_orthogonal_space} and Theorem \ref{theorem_m=0} leaves us with
		\begin{align*}
			& \langle L_1\overline{f},\overline{f}\rangle_{1,n} + \frac{1}{n-1}\langle \overline{f},\overline{f}\rangle_{1,n} \\ 
			& \quad = \ -\frac{1}{n(n-1)}\int_{-1}^1 \mathcal{A}_1\eta(t)^{n-2}(1-t^2)^{\frac{n-1}{2}}\left(1+{\overline{g}^\perp}'(t)\right)^2dt \\
			& \quad \quad + \frac{1}{n(n-1)}\mathcal{A}_1\eta(t)^{n-3}(1-t^2)^{\frac{n-3}{2}}k_0(t)\left(t+\overline{g}^\perp(t)\right)^2dt \\
			& \quad \quad + \frac{1}{n-1}\left(\langle t,t\rangle_{0,n} + \langle \overline{g}^\perp,\overline{g}^\perp\rangle_{0,n}\right) \\
			& \quad \quad -\frac{n-2}{n(n-1)} \int_{-1}^1 \mathcal{A}_2\eta(t)\mathcal{A}_1\eta(t)^{n-3} (1-t^2)^{\frac{n-3}{2}} \overline{f}(t)^2dt \\
			& \quad \leq \frac{1}{n(n-1)}\int_{-1}^1\frac{\mathcal{A}_2\eta(t)\mathcal{A}_1\eta(t)^{n-2}}{\eta(t)}(1-t^2)^{\frac{n-3}{2}}t^2dt \\
			& \quad \quad -\frac{n-2}{n(n-1)} \int_{-1}^1 \mathcal{A}_2\eta(t)\mathcal{A}_1\eta(t)^{n-3} (1-t^2)^{\frac{n-3}{2}} \overline{f}(t)^2dt. 
		\end{align*}
		Now by the assumption on $\overline{f}$ we have
		\begin{align*}
			&\int_{-1}^1\frac{\mathcal{A}_2\eta(t)\mathcal{A}_1\eta(t)^{n-2}}{\eta(t)}(1-t^2)^{\frac{n-3}{2}}t^2dt	\\
			& \quad = \int_{-1}^1\frac{\mathcal{A}_2\eta(t)\mathcal{A}_1\eta(t)^{n-2}}{\eta(t)}(1-t^2)^{\frac{n-2}{2}}t\overline{f}(t)dt,
		\end{align*}
		which proves the statement.
	\end{proof}

	\begin{remark}\label{remark_n=2}
		Clearly, for $n=2$ the right hand side of the inequality (\ref{ineq_odd_case}) is positive for generic odd $\overline{f}\neq 0$. Therefore Theorem \ref{theorem_m=1} does not follow in the case $n=2$, which is why we need to assume $n\geq 3$ for the present section.  Nonetheless, Theorem \ref{theorem_m=1} holds for $n=2$ as the known logarithmic Brunn-Minkowski inequality in $\R^2$ implies the local one. Also van Handel's result for zonoids in \cite{vanhandel} shows the local logarithmic Brunn-Minkowski inequality in $\R^2$ as every convex body in $\R^2$ is a zonoid. This shows that our bound is not sharp when $n=2$. The reason for this is that we cannot chose $\overline{g}^\perp=0$ in Lemma \ref{odd_case} for $n=2$ as for $\overline{f}:(-1,1)\to \R$ given by $\overline{f}(t)=t(1-t^2)^{-\frac{1}{2}}$ we have $\overline{f}\notin H_1$ whenever $n\in\{2,3\}$. 
	\end{remark}
	
	In order to prove Theorem \ref{theorem_m=1} for odd $\overline{f}$, we need to prove the following.
	
	\begin{lemma}\label{lemma_CS_odd}
		Assume $n\geq 3$. Then for any $0\neq\overline{f}\in L^2([-1,1],w_{1,n}(t)dt)$ we have	
		\begin{align*}
			&\left(\int_{-1}^1\frac{\mathcal{A}_2\eta(t)\mathcal{A}_1\eta(t)^{n-2}}{\eta(t)}(1-t^2)^{\frac{n-2}{2}}t\overline{f}(t)dt\right)^2 \\
			& \quad \leq  (n-2) \left(\int_{-1}^1 \mathcal{A}_2\eta(t)\mathcal{A}_1\eta(t)^{n-3} (1-t^2)^{\frac{n-3}{2}} \overline{f}(t)^2dt\right)\\
			& \hspace{5cm} \left(\int_{-1}^1\frac{\mathcal{A}_2\eta(t)\mathcal{A}_1\eta(t)^{n-2}}{\eta(t)}(1-t^2)^{\frac{n-3}{2}}t^2dt\right).
		\end{align*}
	\end{lemma}
	
	\begin{proof}
		By the Cauchy-Schwarz inequality we have
		\begin{align*}
			&\left(\int_{-1}^1 \frac{\mathcal{A}_2\eta(t)\mathcal{A}_1\eta(t)^{n-2}}{\eta(t)}(1-t^2)^{\frac{n-2}{2}} t\overline{f}(t)dt\right)^2  \\
			& \quad \leq \left(\int_{-1}^1 \mathcal{A}_2\eta(t)\mathcal{A}_1\eta(t)^{n-3}(1-t^2)^{\frac{n-3}{2}} \overline{f}(t)^2dt\right)\\
			& \hspace{5cm}\left(\int_{-1}^1 \frac{\mathcal{A}_2\eta(t)\mathcal{A}_1\eta(t)^{n-1}}{\eta(t)^2}(1-t^2)^{\frac{n-1}{2}}t^2dt\right).
		\end{align*}
		Now an easy calculation using the fact that $\eta$ is even shows that for all $t\in[-1,1]$ we have
		\begin{align}\label{RK-HK_ineq}
			\eta(t)-\mathcal{A}_1\eta(t)(1-t^2)=t\int_0^t\mathcal{A}_2\eta(s)ds\geq 0.
		\end{align}
		With that we get
		\begin{align*}
			\int_{-1}^1 \frac{\mathcal{A}_2\eta(t)\mathcal{A}_1\eta(t)^{n-1}}{\eta(t)^2}(1-t^2)^{\frac{n-1}{2}}t^2dt\leq \int_{-1}^1 \frac{\mathcal{A}_2\eta(t)\mathcal{A}_1\eta(t)^{n-2}}{\eta(t)}(1-t^2)^{\frac{n-3}{2}}t^2dt.
		\end{align*}
		This yields
		\begin{align*}
			&\left(\int_{-1}^1 \frac{\mathcal{A}_2\eta(t)\mathcal{A}_1\eta(t)^{n-2}}{\eta(t)}(1-t^2)^{\frac{n-2}{2}} t\overline{f}(t)dt\right)^2  \\
			& \quad \leq  \left(\int_{-1}^1 \mathcal{A}_2\eta(t)\mathcal{A}_1\eta(t)^{n-3}(1-t^2)^{\frac{n-3}{2}} \overline{f}(t)^2dt\right)\\
			& \quad \hspace{5cm}  \left(\int_{-1}^1 \frac{\mathcal{A}_2\eta(t)\mathcal{A}_1\eta(t)^{n-2}}{\eta(t)}(1-t^2)^{\frac{n-3}{2}}t^2dt\right).
		\end{align*}
		As $1 \leq n-2$ by the assumption $n\geq 3$, this proves the statement.
	\end{proof}
	
	\begin{remark}\label{remark_CS_odd}
		The identity (\ref{RK-HK_ineq}) can also be seen more geometrically. By \cite[Theorem 3.3]{braunerhofstaetterortegamoreno2}, $\mathcal{A}_1\eta(t)\sqrt{1-t^2}$ is precisely the radius of the circle $C_t=\nu_K^{-1}(\{x\in S^{n-1}:x_n=t\})$ where $\nu_K:\partial K\to S^{n-1}$ denotes the Gauss map of $K$. Now for any $y\in S^{n-2}$ write $\nu_K^{-1}(\sqrt{1-t^2}y,t)=(\mathcal{A}_1\eta(t)\sqrt{1-t^2}y,p_n)$ for some $p_n\in\R$. Then
		\begin{align*}
			\eta(t)=\langle (\mathcal{A}_1\eta(t)\sqrt{1-t^2}y,p_n),(\sqrt{1-t^2}y,t)\rangle= \mathcal{A}_1\eta(t)(1-t^2)+p_nt.
		\end{align*}
		The fact that $K$ contains the origin in the interior directly yields $p_nt\geq0$. To get the quantitative statement, denote by $P=\spann(y,e_n)$ and assume without loss of generality $t\geq 0$. Then observe that
		\begin{align*}
			\int_0^t \mathcal{A}_2\eta(s)ds=& \ \frac{1}{2}\int_{-t}^t \frac{\mathcal{A}_2\eta(s)}{\sqrt{1-s^2}}\sqrt{1-s^2}ds \\
			=& \ \frac{1}{2}V_P(\{x\in K:|x_n|\leq p_n\}\cap P,[-1,1]y) \\
			=& \ p_n,
		\end{align*}
		where $V_P$ denotes the mixed volume in $P$.
	\end{remark}

	Finally we can prove Theorem \ref{theorem_m=1} with a homotopy argument as in Section \ref{section_m=0}.
	
	\begin{proof}[Proof of Theorem \ref{theorem_m=1}]
		Without loss of generality we can assume that $\eta$ is analytic. Now for $s\in [0,1]$ define $\eta_s(t)=1-s+s\eta(t)$ and let $L_{1,s}$ be the operator $L_1$ defined by $\eta_s$ as well as $\langle\cdot,\cdot\rangle_{1,n,s}$ the respective inner product. Assume boundary conditions (\ref{boundary_conditions_arbitrary_m}) for $m=1$ and consider the function
		\begin{align}\label{definition_F(s)_m=1}
			F(s)= \sup_{0\neq\overline{f}\in H_1 \text{ odd}} \frac{\langle L_{1,s}\overline{f},\overline{f}\rangle_{1,n,s}}{\langle \overline{f},\overline{f}\rangle_{1,n,s}}-\sup_{\substack{0\neq\overline{f}\in H_1 \text{ even}\\ \overline{f}\perp 1}} \frac{\langle L_{1,s}\overline{f},\overline{f}\rangle_{1,n,s}}{\langle \overline{f},\overline{f}\rangle_{1,n,s}}.
		\end{align}
		Now if $n\geq 3$ then Lemma \ref{odd_case} and Lemma \ref{lemma_CS_odd} yield that the first summand is $\leq -\frac{1}{n-1}$ for all $s\in[0,1]$. The same holds for $n=2$ by Remark \ref{remark_n=2}. For $s=0$ consider
		\begin{align*}
			L_{1,0}\overline{f}=\mu\overline{f}
		\end{align*}
		with boundary conditions (\ref{boundary_conditions_arbitrary_m}) for $m=1$. Explicitly
		\begin{align*}
			L_{1,0}\overline{f}(t)=\frac{1}{(n-1)(1-t^2)^{\frac{n-1}{2}}}\frac{d}{dt}(1-t^2)^{\frac{n+1}{2}}\overline{f}'(t)
		\end{align*}
		and hence the respective eigenfunctions are given by the Gegenbauer polynomials $C_k^{\left(\frac{n}{2}\right)}$. The eigenvalues are precisely $\mu_k=-\frac{k(k+n)}{n-1}$. The eigenvalue $\mu_0=0$ corresponds to the constant eigenfunction $C_0^{\left(\frac{n}{2}\right)}$ and therefore
		\begin{align*}
			F(0)=\mu_1 - \mu_2 > 0.
		\end{align*}
		Furthermore for all $s\in[0,1]$, constant functions are eigenfunctions to $L_{1,s}$ with eigenvalue $0$. The same homotopy argument as in the proof of Theorem \ref{theorem_m=0} then shows that the assumption $F(1)<0$ leads to a contradiction. This proves (\ref{inner_product_ineq_m=1}). 
	\end{proof}

	\subsection{The operator on $E_m$ for $m\geq 2$}\label{section_m>1}
	
	Assume $m\geq 2$ for the whole section. Our aim is to prove the following theorem.
	
	\begin{theorem}\label{theorem_m>1}
		For all $\overline{f}\in H_m$ with boundary conditions (\ref{boundary_conditions_arbitrary_m}) we have
		\begin{align*}
			\langle L_m\overline{f},\overline{f}\rangle_{m,n}+\frac{1}{n-1}\langle \overline{f},\overline{f}\rangle_{m,n}\leq 0.
		\end{align*}
	\end{theorem}
	
	\begin{lemma}\label{even_case}
		If $\overline{f} \in H_m$ is even and satisfies the Neumann boundary conditions (\ref{boundary_conditions_arbitrary_m}) then
		\begin{align*}
			& \langle L_m\overline{f},\overline{f}\rangle_{m,n} + \frac{1}{n-1}\langle \overline{f},\overline{f}\rangle_{m,n} \nonumber \\ 
			& \quad \leq \frac{1}{n-1}\frac{\left(\int_{-1}^1 \mathcal{A}_2\eta(t)\mathcal{A}_1\eta(t)^{n-2}(1-t^2)^{\frac{m}{2}+\frac{n-3}{2}} \overline{f}(t)dt\right)^2}{\int_{-1}^1 \mathcal{A}_2\eta(t)\mathcal{A}_1\eta(t)^{n-2}\eta(t)(1-t^2)^{\frac{n-3}{2}}dt} \\
			& \quad \quad -\frac{m(m+n-3)}{n(n-1)}\int_{-1}^1 \mathcal{A}_2\eta(t)\mathcal{A}_1\eta(t)^{n-3}(1-t^2)^{m+\frac{n-5}{2}} \overline{f}(t)^2dt. \nonumber
		\end{align*}
	\end{lemma}
	
	\begin{proof}
		As in Lemma \ref{odd_case} write $\overline{f}(t)=(1-t^2)^{-\frac{m}{2}}\left(\eta(t)+\overline{g}^\perp(t)\right)$ with $\langle \overline{g}^\perp,\eta\rangle_{0,n}=0$. By orthogonality, we have
		\begin{align*}
			\langle \overline{f},\overline{f}\rangle_{m,n}=& \langle \eta,\eta\rangle_{0,n} + \langle \overline{g}^\perp,\overline{g}^\perp\rangle_{0,n}.
		\end{align*}
		The same calculation as in Lemma \ref{odd_case} using Lemma \ref{lemma_orthogonal_space} and Theorem \ref{theorem_m=0} leaves us with
		\begin{align*}
			& n\langle L_m\overline{f},\overline{f}\rangle_{m,n} + \frac{n}{n-1}\langle \overline{f},\overline{f}\rangle_{m,n} \\ 
			& \quad \leq \frac{n}{n-1}\int_{-1}^1 \mathcal{A}_2\eta(t)\mathcal{A}_1\eta(t)^{n-2}(1-t^2)^{\frac{m}{2}+\frac{n-3}{2}} \overline{f}(t)dt \\
			& \quad \quad  -\frac{m(m+n-3)}{n-1}\int_{-1}^1 \mathcal{A}_2\eta(t)\mathcal{A}_1\eta(t)^{n-3}(1-t^2)^{m-1+\frac{n-3}{2}}\overline{f}(t)^2dt.
		\end{align*}
		Now, by the choice of $\overline{f}$ we have
		\begin{align*}
			\int_{-1}^1 \mathcal{A}_2\eta(t)\mathcal{A}_1\eta(t)^{n-2}(1-t^2)^{\frac{m}{2}+\frac{n-3}{2}} \overline{f}(t)dt = \int_{-1}^1 \mathcal{A}_2\eta(t)\mathcal{A}_1\eta(t)^{n-2}\eta(t)(1-t^2)^{\frac{n-3}{2}}dt,
		\end{align*}
		which proves the statement.
	\end{proof}
	
	Now in order to prove our theorem in the even case, we need to prove the following.
	
	\begin{lemma}\label{lemma_CS_even}
		For any $0\neq\overline{f}\in L^2([-1,1],w_{m,n}(t)dt)$ and $m\geq 2$ we have
		\begin{align*}
			&\left(\int_{-1}^1 \mathcal{A}_2\eta(t)\mathcal{A}_1\eta(t)^{n-2}(1-t^2)^{\frac{m}{2}+\frac{n-3}{2}} \overline{f}(t)dt\right)^2  \\
			& \quad \leq \frac{m(m+n-3)}{n}\left(\int_{-1}^1 \mathcal{A}_2\eta(t)\mathcal{A}_1\eta(t)^{n-3}(1-t^2)^{m+\frac{n-5}{2}} \overline{f}(t)^2dt\right)\\
			& \quad \hspace{5cm}  \left(\int_{-1}^1 \mathcal{A}_2\eta(t)\mathcal{A}_1\eta(t)^{n-2}\eta(t)(1-t^2)^{\frac{n-3}{2}}dt\right).
		\end{align*}
	\end{lemma}
	
	\begin{proof}
		By the Cauchy-Schwarz inequality we have
		\begin{align*}
			&\left(\int_{-1}^1 \mathcal{A}_2\eta(t)\mathcal{A}_1\eta(t)^{n-2}(1-t^2)^{\frac{m}{2}+\frac{n-3}{2}} \overline{f}(t)dt\right)^2  \\
			& \quad \leq \left(\int_{-1}^1 \mathcal{A}_2\eta(t)\mathcal{A}_1\eta(t)^{n-3}(1-t^2)^{m+\frac{n-5}{2}} \overline{f}(t)^2dt\right)\left(\int_{-1}^1 \mathcal{A}_2\eta(t)\mathcal{A}_1\eta(t)^{n-1}(1-t^2)^{\frac{n-1}{2}}dt\right).
		\end{align*}
		Now integrating by parts two times yields
		\begin{align*}
			&\int_{-1}^1 \mathcal{A}_2\eta(t)\mathcal{A}_1\eta(t)^{n-1}(1-t^2)^{\frac{n-1}{2}}dt \\
%			& \quad = \int_{-1}^1 \eta''(t)\mathcal{A}_1\eta(t)^{n-1}(1-t^2)^{\frac{n+1}{2}}dt + \int_{-1}^1 \mathcal{A}_1\eta(t)^n(1-t^2)^{\frac{n-1}{2}}dt \\
%			& \quad = \int_{-1}^1 \eta'(t)\mathcal{A}_1\eta(t)^{n-2}t(1-t^2)^{\frac{n-1}{2}}\left( (n-1)\mathcal{A}_2\eta(t)+2\mathcal{A}_1\eta(t)\right)dt \\
%			& \quad \quad + \int_{-1}^1 \mathcal{A}_1\eta(t)^n(1-t^2)^{\frac{n-1}{2}}dt \\
%			& \quad = -(n-1)\int_{-1}^1 \mathcal{A}_2\eta(t)\mathcal{A}_1\eta(t)^{n-1}(1-t^2)^{\frac{n-1}{2}}dt \\
%			& \quad \quad + (n-1)\int_{-1}^1 \mathcal{A}_2\eta(t)\mathcal{A}_1\eta(t)^{n-2}\eta(t)(1-t^2)^{\frac{n-1}{2}}dt \\
%			& \quad \quad + \int_{-1}^1 t\eta'(t)\mathcal{A}_1\eta(t)^{n-1}(1-t^2)^{\frac{n-1}{2}}dt + \int_{-1}^1 \mathcal{A}_1\eta(t)^{n-1}\eta(t)(1-t^2)^{\frac{n-1}{2}}dt \\
			& \quad = -(n-1)\int_{-1}^1 \mathcal{A}_2\eta(t)\mathcal{A}_1\eta^{n-1}(1-t^2)^{\frac{n-1}{2}}dt \\
			& \hspace{4cm} + (n-1)\int_{-1}^1 \mathcal{A}_2\eta(t)\mathcal{A}_1\eta(t)^{n-2}\eta(t)(1-t^2)^{\frac{n-3}{2}}dt
		\end{align*}
		and therefore
		\begin{align}\label{equality_Kubota}
			\int_{-1}^1 \mathcal{A}_2\eta(t)\mathcal{A}_1\eta(t)^{n-1}(1-t^2)^{\frac{n-1}{2}}dt=\frac{n-1}{n}\int_{-1}^1 \mathcal{A}_2\eta(t)\mathcal{A}_1\eta(t)^{n-2}\eta(t)(1-t^2)^{\frac{n-3}{2}}dt.
		\end{align}
		This yields
		\begin{align*}
			&\left(\int_{-1}^1 \mathcal{A}_2\eta(t)\mathcal{A}_1\eta(t)^{n-2}(1-t^2)^{\frac{m}{2}+\frac{n-3}{2}} \overline{f}(t)dt\right)^2  \\
			& \quad \leq  \frac{n-1}{n}\left(\int_{-1}^1 \mathcal{A}_2\eta(t)\mathcal{A}_1\eta(t)^{n-3}(1-t^2)^{m+\frac{n-5}{2}} \overline{f}(t)^2dt\right)\\
			& \quad \hspace{5cm}  \left(\int_{-1}^1 \mathcal{A}_2\eta(t)\mathcal{A}_1\eta(t)^{n-2}\eta(t)(1-t^2)^{\frac{n-3}{2}}dt\right).
		\end{align*}
		As $\frac{n-1}{n} < \frac{m(m+n-3)}{n}$ for $m\geq 2$ this proves the statement.
	\end{proof} 
	
	\begin{remark}
		The equality (\ref{equality_Kubota}) can also be seen as follows. Denote by $K^{n+1}$ the convex body in $\R^{n+1}$ with support function $h_{K^{n+1}}(x)=\eta(x_{n+1})$. Then the equality is precisely
		\begin{align*}
			\frac{n+1}{\omega_n}V(K^{n+1}[n],\mathbb{D}^n)=\frac{n-1}{\omega_{n-1}}\vol(K),
		\end{align*}
		where $\omega_k$ denotes the volume of the euclidean unit ball in $\R^k$ and $\mathbb{D}^n$ denotes the euclidean unit ball in $\R^n\times\{0\}\subset \R^{n+1}$. This equality is known to hold by a Kubota-type formula established by Hug, Mussnig and Ulivelli in \cite[Theorem 1.1]{hugmussnigulivelli}.
	\end{remark}

	With the well known homotopy argument we can now prove Theorem \ref{theorem_m>1}.
	
	\begin{proof}[Proof of Theorem \ref{theorem_m>1}]
		As before, assume without loss of generality that $\eta$ is analytic. For $s\in [0,1]$ define $\eta_s(t)=1-s+s\eta(t)$ and let $L_{m,s}$ be the operator $L_m$ defined by $\eta_s$ as well as $\langle\cdot,\cdot\rangle_{m,n,s}$ the respective inner product. Assume boundary conditions (\ref{boundary_conditions_arbitrary_m}) and consider the function
		\begin{align}\label{definition_F(s)_m>1}
			F(s)= \sup_{0\neq\overline{f}\in H_m \text{ even}} \frac{\langle L_{m,s}\overline{f},\overline{f}\rangle_{m,n,s}}{\langle \overline{f},\overline{f}\rangle_{m,n,s}}-\sup_{0\neq\overline{f}\in H_m \text{ odd}} \frac{\langle L_{m,s}\overline{f},\overline{f}\rangle_{m,n,s}}{\langle \overline{f},\overline{f}\rangle_{m,n,s}}.
		\end{align}
		Now Lemma \ref{even_case} and Lemma \ref{lemma_CS_even} yield that the first summand is $\leq -\frac{1}{n-1}$ for all $s\in[0,1]$. Now using the fact, that for $m\geq 2$ all eigenvalues of $L_{m,0}$ are $\leq -\frac{1}{n-1}$ the known homotopy argument proves the statement.
	\end{proof}
	
	We can now prove the desired eigenvalue bound on the $\mathcal{D}_m$ for $m\geq 1$ completing the proof of Theorem \ref{spectral_gap_theorem}.
	
	\begin{theorem}\label{theorem_m>0}
		For any $m\in\N$ and any eigenfunction $f$ of $\mathcal{D}_m$ with eigenvalue $\lambda$ that is orthogonal to linear functions we have 
		\begin{align*}
			\lambda \leq -\frac{1}{n-1}.
		\end{align*}
	\end{theorem}
	
	\begin{proof}
		Assume that $f$ is an even eigenfunction of $\mathcal{D}_m$ to the eigenvalue $\lambda$. By elliptic regularity, this means that $f$ is smooth, especially $C^2$. Then we can write $f=\sum_{i=1}^{N}h_i\otimes \overline{f}_i$ for pairwise orthogonal $h_i\in\mathcal{H}_m^{n-1}$ and $\overline{f}\in C^2(-1,1)$. Since $f$ is $C^2$ Lemma \ref{lemma_Cinfty} and Lemma \ref{lemma_explicit_op} imply that
		\begin{align*}
			\lambda\langle f,f\rangle=\langle \mathcal{D}_mf,f\rangle = \sum_{i=1}^N \langle h_i,h_i\rangle_{L^2(S^{n-2})} \langle L_m\overline{f}_i,\overline{f}_i\rangle_{m,n}.
		\end{align*}
		Furthermore, all $f_i$ satisfy (\ref{dirichlet_boundary_conditions}) and (\ref{neumann_boundary_conditions}). If $m=1$ then the $h_i$ are linear functions and the condition $\langle f,h_i\rangle=0$ for all $i=1,\dots,N$ yields that the $\overline{f}_i$ are orthogonal to constants. We claim that for all $i=1,\dots N$, we have $\overline{f}_i\in H_m$ and satisfies the Neumann boundary conditions (\ref{boundary_conditions_arbitrary_m}). Then Theorem \ref{theorem_m=1} and Theorem \ref{theorem_m>1} show that
		\begin{align*}
			\langle L_m\overline{f}_i,\overline{f}_i\rangle_{m,n} + \frac{1}{n-1}\langle \overline{f}_i,\overline{f}_i\rangle_{m,n} \leq 0.
		\end{align*}
		and therefore $\lambda\leq -\frac{1}{n-1}$. 
		
		To prove the claim, assume $\overline{f}\in C^2(-1,1)$ satisfies (\ref{dirichlet_boundary_conditions}) and (\ref{neumann_boundary_conditions}), i.e.
		\begin{align}\label{dirichlet_bc2}
			\limsup\limits_{t\to\pm 1}(1-t^2)^{\frac{m-1}{2}}\left|\overline{f}(t)\right|<\infty
		\end{align}
		and
		\begin{align}\label{neuman_bc2}
			\limsup\limits_{t\to\pm 1} (1-t^2)^{\frac{m+1}{2}}\left|\overline{f}'(t)\right|<\infty
		\end{align}
		Then (\ref{dirichlet_bc2}) directly implies that $\overline{f}\in L^2([-1,1],(1-t^2)^{m+\frac{n-3}{2}})$. Furthermore by (\ref{neuman_bc2}) we have
		\begin{align*}
			&\int_{-1}^1 (1-t^2)^{m+\frac{n-1}{2}}\overline{f}'(t)^2dt \leq \left(\sup_{t\in[-1,1]}(1-t^2)^{\frac{m+1}{2}}\left|\overline{f}'(t)\right|\right)^2\int_{-1}^1(1-t^2)^{\frac{n-3}{2}}dt<\infty
		\end{align*}
		and therefore $\overline{f}\in H_m$. Finally for any $u\in H_m$ we have
		\begin{align*}
			(1-t^2)^{m+\frac{n-1}{2}}\overline{f}'(t)u(t)\big|_{t=-1}^1=(1-t^2)^{\frac{m+1}{2}}\overline{f}'(t) (1-t^2)^{\frac{m+n-2}{2}}u(t)\big|_{t=-1}^1.
		\end{align*}
		By Lemma \ref{lemma_compact_embedding}(ii) $\lim\limits_{t\to\pm 1}(1-t^2)^{\frac{m+n-2}{2}}u(t)=0$. Together with (\ref{neuman_bc2}) this proves, that $\overline{f}$ satisfies the Neumann boundary conditions (\ref{boundary_conditions_arbitrary_m}) and therefore completes the proof.
	\end{proof}
	
	\section{Auxiliary results}\label{section_aux_results}
	
	In this section assume $n\geq 3$. For any $K\in\Kc(\R^n)$ denote
	\begin{align*}
		L_K(t)=\inf\left\{tp_n:p\in\partial K,(\sqrt{1-t^2}y,t)\text{ is normal to }K\text{ at }p\text{ for some }y\in S^{n-1}\right\}.
	\end{align*}
	By Remark \ref{remark_CS_odd}, if $K$ is of class $C^2_+$ with $h_K(x)=\eta(x_n)$ then $L_K(t)=t\int_0^t\mathcal{A}_2\eta(t)dt$. Furthermore for $m>0$ let
	\begin{align*}
		c_m(K)= \begin{cases}
			\frac{n-3}{n-1}+\frac{\int_{S^{n-1}}\frac{x_n^2L_K(x_n)}{h_K(x)}dS_K(x)}{\int_{S^{n-1}}\frac{x_n^2}{h_K(x)}dS_K(x)} & \text{for } m=1 \\
			\frac{m(m+n-3)}{n-1}-1 & \text{for } m\geq 2.
		\end{cases}
	\end{align*}
	If $K$ is origin symmetric with $0$ in the interior, $L_K(t)\geq 0$ for all $t$ with strict inequality at least at some point. Therefore $c_m(K) >0$ for all $m>0$. Now we have the following strengthened local logarithmic Brunn-Minkowski inequality. 
	
	\begin{theorem}\label{stronger_loc_log_bm}
		Let $K$ be an origin symmetric convex body of revolution with nonempty interior and $f\in C(S^{n-1})$ a difference of support functions such that
		\begin{align*}
			\int_{S^{n-1}} \frac{xf(x)}{h_K(x)}dS_K(x)=0.
		\end{align*}
		Then
		\begin{align*}
			&V(f,f,K[n-2])-\frac{n}{n-1}\frac{V(f,K[n-1])^2}{\vol(K)}+\frac{1}{ n(n-1)}\int_{S^{n-1}}\frac{f^2}{h_K}dS_K \\
			& \quad \leq  -\frac{1}{n} \sum_{m=1}^\infty c_m(K) \int_{S^{n-1}} \frac{(\pi_mf)^2}{h_K}dS_K.
		\end{align*}
	\end{theorem}
	
	\begin{proof}
		First, we assume that $K$ is of the class $C_+^2$ and $f\in C^2(S^{n-1})$. Set $\tilde{f}=f-\frac{\langle f,h_K\rangle}{\langle h_K,h_K\rangle}h_K$. Then with $f_m=\pi_mf$ we have
		\begin{align*}
			\tilde{f}=f_0-\frac{\langle f_0,h_k\rangle}{\langle h_K,h_K\rangle}h_K+\sum_{m=1}^\infty f_m,
		\end{align*}
		where the sum is orthogonal with respect to $\langle\cdot,\cdot\rangle$. By the same calculation as in the proof of Theorem \ref{local_log_bm_zonal} together with Theorem \ref{theorem_m=0} we get
		\begin{align*}
			&V(f,f,K[n-2])-\frac{n}{n-1}\frac{V(f,K[n-1])^2}{\vol(K)}+\frac{1}{ n(n-1)}\int_{S^{n-1}}\frac{f^2}{h_K}dS_K \\
			& \quad = \langle \mathcal{D}_K\tilde{f},\tilde{f}\rangle + \frac{1}{n-1}\langle \tilde{f},\tilde{f}\rangle \\
			& \quad \leq \sum_{m=1}^\infty \left(\langle \mathcal{D}_mf_m,f_m\rangle + \frac{1}{n-1}\langle f_m,f_m\rangle\right).
		\end{align*}
		Now write $f_m=\sum_{i=1}^Nh_i\otimes \overline{f}_i$ and $h_K(x)=\eta(x_n)$ as before. First assume that all the $\overline{f}_i$ are odd if $m=1$ and that all the $\overline{f}_i$ are even if $m\geq 2$. Then analogously to the proof of Theorem \ref{theorem_m>0} we get
		\begin{align*}
			&\langle \mathcal{D}_mf_m,f_m\rangle + \frac{1}{n-1}\langle f_m,f_m\rangle \\
			& \quad = \sum_{i=1}^N \langle h_i,h_i\rangle_{L^2(S^{n-2})} \left(\langle L_m\overline{f}_i,\overline{f}_i\rangle_{m,n}+\frac{1}{n-1}\langle\overline{f}_i,\overline{f}_i\rangle_{m,n}\right) \\
			& \quad \leq  -\frac{c_m(K)}{n}\sum_{i=1}^N \langle h_i,h_i\rangle_{L^2(S^{n-2})}\int_{-1}^1 \mathcal{A}_2\eta(t)\mathcal{A}_1\eta(t)^{n-3}(1-t^2)^{m+\frac{n-5}{2}}\overline{f}_i(t)^2dt \\
			& \quad = -\frac{c_m(K)}{n}\sum_{i=1}^N \int_{S^{n-1}} \frac{h_i(x_1,\dots,x_{n-1})^2\overline{f}_i(x_n)^2}{(1-x_n^2)\mathcal{A}_1(x_n)}dS_K(x).
		\end{align*}
		Here we used Lemma \ref{odd_case} together with the proof of Lemma \ref{lemma_CS_odd} for $m=1$ and Lemma \ref{even_case} together with the proof of Lemma \ref{lemma_CS_even} for even $m\geq 2$ as well as the the definition of $c_m(K)$. Using the inequality in (\ref{RK-HK_ineq}) we get
		\begin{align}\label{stronger_ev_ineq}
			\langle \mathcal{D}_mf_m,f_m\rangle + \frac{1}{n-1}\langle f_m,f_m\rangle \leq -\frac{c_m(K)}{n}\sum_{i=1}^N \int_{S^{n-1}} \frac{h_i(x_1,\dots,x_{n-1})^2\overline{f}_i(x_n)^2}{h_K(x)}dS_K(x).
		\end{align}
		Now the same homotopy argument as in the proof of Theorem \ref{theorem_m=1} and Theorem \ref{theorem_m>1} removes the parity assumptions on the $\overline{f}_i$. Then (\ref{stronger_ev_ineq}) together with the above shows the stated inequality for $f\in C^2(S^{n-1})$ and $K$ of class $C^2_+$. Finally for every sequence of convex bodies $(K_k)_{k\in\N}$ with $K_k\to K$ in the Hausdorff topology, we have $\lim\limits_{k\to\infty} L_{K_k}(t)\geq L_K(t)$ for all $t\in[-1,1]$ as a consequence of the upper semi-continuity of the faces of a convex body. Consequently $\lim\limits_{k\to\infty}c_m(K_k)\leq c_m(K)$ for all $m>0$. Then together with the continuity of $\pi_m$, the general statement follows by approximation.
	\end{proof}
	
	To treat the equality cases, we will use the stability terms on the right hand side in Theorem \ref{stronger_loc_log_bm} to reduce the problem to the case of two bodies of revolution. The same method also yields a strict inequality of the eigenvalues of $\mathcal{D}_m$ for $m>0$. 
	
	\begin{proof}[Proof of Theorem \ref{equality_cases}]
		Assume that the pair $(K,L)$ satisfies the assumptions of Theorem \ref{local_log_bm_zonal} such that on $S^{n-1}$ the measure $\mathcal{H}^{n-1}$ is absolutely continuous with respect to $S_K$ and that equality holds in (\ref{loc_log_bm_ineq}). Then Theorem \ref{stronger_loc_log_bm} yields that
		\begin{align*}
			0\leq -\sum_{m=1}^\infty c_m(K) \int_{S^{n-1}} \frac{(\pi_mh_L)^2}{h_K}dS_K. 
		\end{align*}
		As each summand on the right hand side is nonpositive, this implies 
		\begin{align*}
			\int_{S^{n-1}} \frac{(\pi_mh_L)^2}{h_K}dS_K=0
		\end{align*}
		for all $m\geq 1$. This precisely means that $S_K(\supp \pi_mh_L)=0$ and hence by our assumption $\mathcal{H}^{n-1}(\supp \pi_mH_L)=0$. As $\pi_mh_L$ is continuous, we then have $\pi_mh_L=0$ for all $m\geq 1$. Therefore $h_L$ is zonal and hence $L$ is a body of revolution. 
	\end{proof}

	Theorem \ref{stronger_loc_log_bm} gives stability terms for the local logarithmic Brunn-Minkowski inequality on all the $E_m$ for $m>0$. It remains open, whether there are nontrivial equality cases for two convex bodies of revolution and whether there is a stronger eigenvalue bound on $\mathcal{D}_0$. Consider the following problem.
	
	\begin{problem}\label{local_lp_bm_problem_p<0}
		For any origin symmetric convex body of revolution $K\subset\R^n$, determine if there is a constant $p(K,n)<0$ such that the local $L^p$-Brunn-Minkowski inequality (\ref{loc_log_bm_ineq}) holds true for all $p\geq p(K,n)$ and all convex bodies $L$ satisfying (\ref{center_of_mass_condition}).
	\end{problem}
	
	By the same argument as in the proof of Theorem \ref{spectral_gap_theorem}, this corresponds directly to finding $p(K,n)$ such that any eigenvalue $\lambda\notin\{0,1\}$ of $\mathcal{D}_K$ satisfies  
	\begin{align*}
		\lambda\leq - \frac{1-p(K,n)}{n-1}.
	\end{align*}
	
	The following corollary additionally motivates the work on a stronger eigenvalue bound on $\mathcal{D}_0$ in order to treat Problem \ref{local_lp_bm_problem_p<0}.
		
	\begin{cor}
		Let $K$ be an origin symmetric convex bodies of revolution with nonempty interior and $p\geq -(n-1)\min\{c_1(K),c_2(K)\}$. Let $L$ be a convex body such that (\ref{center_of_mass_condition}) is satisfied. Denote by $\overline{L}$ the $\SO(n-1)$ symmetrization of $L$, i.e. the convex body $\overline{L}$ given by
		\begin{align*}
			h_{\overline{L}}(x)=\int_{\SO(n-1)}h_L(g^{-1}x)dg.
		\end{align*}
		Then if the tuple $(K,\overline{L})$ satisfies the local $L^p$-Brunn-Minkowski inequality (\ref{loc_log_bm_ineq}), so does the tuple $(K,L)$.
	\end{cor}

	\begin{proof}
		Clearly as the $E_m$ are orthogonal with respect to the $L^2(S^{n-1},\mu_K)$-inner product
		\begin{align*}
			\sum_{m=1}^\infty c_m(K) \int_{S^{n-1}} \frac{(\pi_mh_L)^2}{h_K}dS_K\geq  & \ \min\{c_1(K),c_2(K)\}\sum_{m=1}^\infty  \int_{S^{n-1}} \frac{(\pi_mh_L)^2}{h_K}dS_K \\
			=& \ \min\{c_1(K),c_2(K)\}\left(\int_{S^{n-1}} \frac{h_L^2}{h_K}dS_K - \int_{S^{n-1}} \frac{h_{\overline{L}}^2}{h_K}dS_K\right). 
		\end{align*}
		Plugging this in Theorem \ref{stronger_loc_log_bm} proves the statement.
	\end{proof}
	
	\section*{Acknowledgments}
	%%%%%%%%%%%%%%%%%%%%%%%%%%%%%%%%%%%%%%%%%%%%%%%%%%%%%%%%%%%%%%%%%%%%%%%%%%%
	
	The author was supported by the Deutsche Forschungsgesellschaft (DFG) with the project \href{https://gepris.dfg.de/gepris/projekt/520350299}{520350299}. The author wants to thank A. Bernig, L. Brauner, S. Jarohs, S. Kistner, J. Knoerr, O. Ortega-Moreno and J. Ulivelli for many useful discussions.

\end{document}